\documentclass[11pt, reqno]{amsart}

\usepackage{graphicx}
\usepackage{cite}
\usepackage{color}
\usepackage{amsfonts}
\usepackage{amsthm}
\usepackage{url}
\usepackage{amsmath}
\usepackage{amssymb}
\usepackage{enumerate}
\usepackage{subfig}
\usepackage{fixltx2e}
\usepackage{dsfont}
\usepackage{pseudocode}
\usepackage{mathrsfs}
\usepackage[margin=1in]{geometry}

\theoremstyle{plain}
\newtheorem{theorem}{Theorem}[section]
\newtheorem{lemma}[theorem]{Lemma}
\newtheorem{corollary}[theorem]{Corollary}

\newtheorem{proposition}[theorem]{Proposition}

\theoremstyle{definition}
\newtheorem{definition}[theorem]{Definition}
\newtheorem{example}[theorem]{Example}
\newtheorem{remark}[theorem]{Remark}

\theoremstyle{remark}

\newcommand{\EE}{\mathbb{E}}
\newcommand{\R}{\mathbb{R}}

\newcommand{\HS}{\mathrm{HS}}

\newcommand{\Lc}{\Lambda_{\mathbf{c}}}
\newcommand{\Ld}{\Lambda_{\mathbf{d}}}
\newcommand{\Cov}{\operatorname{Cov}}

\newcommand{\Var}{\operatorname{Var}}
\newcommand{\Tr}{\operatorname{tr}}
\renewcommand{\to}{\longrightarrow}
\begin{document}

\title[Euclidean forward-reverse Brascamp-Lieb inequalities]{Euclidean forward-reverse Brascamp-Lieb inequalities:\\finiteness, structure  and extremals}
\author[T.~A.~Courtade]{Thomas~A.~Courtade}
\address{University of California, Berkeley \\ Department of Electrical Engineering and Computer Sciences \\ Berkeley, CA 94720, USA}
\email{courtade@berkeley.edu}

\thanks{This work was  supported by NSF grants CCF-1704967, CCF-0939370 and CCF-1750430.}

\author[J.~Liu]{Jingbo Liu}
\address{Massachusetts Institute of Technology \\ Institute for Data, Systems, and Society (IDSS) \\ Cambridge, MA 02139, USA}
\email{jingbo@mit.edu}


\begin{abstract} 
A new proof is given for the fact that centered gaussian functions saturate the   
Euclidean forward-reverse Brascamp-Lieb inequalities, extending the Brascamp-Lieb and Barthe theorems.  A  duality principle for best constants is also developed, which generalizes the fact that the best constants in the Brascamp-Lieb and Barthe inequalities are equal.  Finally, as the title hints, the main results concerning finiteness, structure and gaussian-extremizability for the Brascamp-Lieb inequality due to  Bennett, Carbery, Christ and Tao are generalized to the setting of the forward-reverse Brascamp-Lieb inequality. 
\end{abstract}

\maketitle 

\section{Introduction and Main Results}
We begin with notation that will prevail throughout.  Let $(E_i)_{1\leq i\leq k}$ and $(E^j)_{1\leq j\leq m}$ be Euclidean spaces, i.e., finite-dimensional Hilbert spaces endowed with Lebesgue measure and the usual  inner product $\langle \cdot, \cdot \rangle$ giving rise to Euclidean length $|\cdot|$.   We write $E_0 = \bigoplus_{i=1}^k E_i$, and let $\pi_{E_i} : E_0 \to E_i$ be the  orthogonal  projection  of $E_0$ onto $E_i$.  %

Let $\mathbf{B}:=(B_{ij})_{1\leq i\leq k,  1\leq j \leq m}$, where each $B_{ij}: E_i \to E^j$ is a bounded linear transformation.   Because it will be referred to frequently,  we define $B_j : E_0 \to E^j$ according to 
$$
B_j x := \sum_{i=1}^k B_{ij}\pi_{E_i}(x),   ~~~~x\in E_0.
$$ 
Note that the collection $(B_j)_{1\leq j\leq m}$ may be regarded as an equivalent characterization of  $\mathbf{B}$.  
We define $\mathbf{B}^*:=(B^*_{ij} )_{1\leq i\leq k,  1\leq j \leq m}$, where $A^*$ denotes the adjoint of $A$. 

We  let $\mathbf{c}:=(c_i)_{1\leq i\leq k}$ and $\mathbf{d}:=(d_j)_{1\leq j\leq m}$ be collections of positive real numbers satisfying 
\begin{align}
\sum_{i=1}^k c_i \dim(E_i) = \sum_{j=1}^m d_j \dim(E^j) , \label{eq:consistency}
\end{align}
and we refer to the triple $(\mathbf{c},\mathbf{d},\mathbf{B})$ as a \emph{datum}.  Finally, $\R^+$ denotes the non-negative real numbers. 

\subsection{The forward-reverse Brascamp-Lieb inequalities}

For a given datum $(\mathbf{c},\mathbf{d},\mathbf{B})$, this paper is concerned with characterizing the best constant $D$ in the following statement:
If measurable functions $f_i : E_i \to \R^+$,  $1\leq i \leq k$ and $g_j : E^j \to \R^+$,  $1\leq j\leq m$  satisfy 
\begin{align}
\prod_{i=1}^k f_i^{c_i}(x_i) \leq \prod_{j=1}^m g_j^{d_j}\left( \sum_{i=1}^k c_i B_{ij} x_i \right)\hspace{1cm}\forall x_i\in E_i,~ 1\leq i\leq k,\label{FRBLhypIntro}
\end{align}
then 
\begin{align}
\prod_{i=1}^k \left( \int_{E_i} f_i  \right)^{c_i} \leq e^D \prod_{j=1}^m \left( \int_{E^j} g_j  \right)^{d_j},\label{FRBLintro}
\end{align}
where the integrals are with respect to Lebesgue measure on the respective spaces. 
To facilitate later referencing, we make a formal definition.
\begin{definition}\label{def:DcdB}
Given a datum $(\mathbf{c},\mathbf{d},\mathbf{B})$, we define $D(\mathbf{c},\mathbf{d},\mathbf{B})$ to be the smallest constant $D$ such that \eqref{FRBLintro} holds for all nonnegative measurable functions satisfying the constraints \eqref{FRBLhypIntro}.
\end{definition}

\begin{remark}
  If \eqref{eq:consistency} does not hold, then dilating all functions by a common factor shows $D(\mathbf{c},\mathbf{d},\mathbf{B}) =+\infty$, motivating the assumption. It is easy to see that $D(\mathbf{c},\mathbf{d},\mathbf{B}) >-\infty$. 
\end{remark}

The above class of  inequalities was introduced by the authors together with Cuff and Verd\'u, and termed  \emph{Forward-Reverse Brascamp-Lieb inequalities} \cite{liu2018forward}.  This choice of terminology reflects the observation that taking $k=1$ and $c_1 = 1$  specializes to the  classical (forward, or direct) Brascamp-Lieb inequalities \cite{brascamp1974general, brascamp1976best, lieb1990gaussian};  on the other hand, taking $m=1$ and $d_1=1$ specializes to the reverse form of the Brascamp-Lieb inequalities introduced by Barthe \cite{barthe1998reverse}.   

The celebrated result of Lieb \cite{lieb1990gaussian} is that in the case $k=c_1=1$, the best constant $D(1,\mathbf{d},\mathbf{B})$ can be computed by considering only centered gaussian functions $f_1, g_1, \dots, g_m$. Likewise, Barthe  showed in  \cite{barthe1998reverse}  that in the case of $m=d_1=1$, the best constant $D(\mathbf{c},1,\mathbf{B})$ can be computed by considering only centered gaussian functions $f_1, \dots, f_k, g_1$.   Barthe also established a remarkable duality between the forward and reverse Brascamp-Lieb inequalities, in the sense that%
\begin{align}
D(\mathbf{c},1,\mathbf{B})=D(1,\mathbf{c},\mathbf{B^*}),\label{dualityBarthe}
\end{align}
 where, by Definition \ref{def:DcdB} applied to the datum $(\mathbf{d},\mathbf{c},\mathbf{B^*})$,  the quantity $D(\mathbf{d},\mathbf{c},\mathbf{B^*})$  denotes the smallest constant $D$ in the inequality 
 \begin{align*}
\prod_{j=1}^m \left( \int_{E^j} g_j  \right)^{d_j} \leq  e^D \prod_{i=1}^k \left( \int_{E_i} f_i  \right)^{c_i},
\end{align*}
 holding for all measurable functions 
 $f_i : E_i \to \R^+$,  $1\leq i\leq k$ and $g_j : E^j \to \R^+$,  $1\leq j \leq  m$  satisfying
\begin{align}
\prod_{j=1}^m g_j^{d_j}(y_j) \leq \prod_{i=1}^k f_i^{c_i} \left( \sum_{j=1}^m d_j B^*_{ij} y_j \right)\hspace{1cm}\forall y_j\in E^j,~1\leq j\leq m. \label{intro:revHyp}
\end{align}

 Perhaps surprisingly, the forward-reverse Brascamp-Lieb inequality suggests that there is no fundamental distinction between the traditional forward and reverse forms of the Brascamp-Lieb inequality. Indeed,  they are each  a particular instance of the inequality \eqref{FRBLintro} under the domination hypothesis \eqref{FRBLhypIntro} for an appropriate choice of datum.    Most importantly, the gaussian saturation property continues to hold for the forward-reverse Brascamp-Lieb inequality, as well as a full-fledged form of the duality relation \eqref{dualityBarthe}.  This both clarifies and unifies the general landscape of Euclidean Brascamp-Lieb-Barthe-type   inequalities and the duality they enjoy.  This is our first main result:
 \begin{theorem}\label{thm:FRBL}
 The quantities $D(\mathbf{c},\mathbf{d},\mathbf{B})$ and $D(\mathbf{d},\mathbf{c},\mathbf{B^*})$ can be computed by considering only centered gaussian functions $(f_i)_{1\leq i\leq k}$ and  $(g_j)_{1\leq j \leq m}$ in their respective definitions.  Moreover, it holds that 
\begin{align}
 D(\mathbf{c},\mathbf{d},\mathbf{B})=D(\mathbf{d},\mathbf{c},\mathbf{B^*}). \label{eq:DualIdentity}
 \end{align}
 \end{theorem}
 \begin{remark}
 The sufficiency of considering gaussian functions for computing the constant $D(\mathbf{c},\mathbf{d},\mathbf{B})$ was already established in our previous work \cite[Theorem 2]{liu2018forward}.  As will be explained in Section \ref{sec:connections},  the gaussian saturation property is closely connected (in fact, formally equivalent) to a result announced  by Barthe and Wolff in the note \cite{barthe2014positivity}, and proved  in their recent followup work \cite{barthe2018positive}. 
The identity \eqref{eq:DualIdentity} has not been previously observed. 
 \end{remark}
\begin{remark}
The identity \eqref{eq:DualIdentity} explains the scaling of each $B_{ij}$ by $c_i$ in \eqref{FRBLhypIntro} and, similarly, the scaling of each $B^*_{ij}$ by $d_j$ in \eqref{intro:revHyp}.  If we were not after \eqref{eq:DualIdentity}, these scalar factors could be absorbed into the maps themselves without affecting the first claim of Theorem  \ref{thm:FRBL}.  %
\end{remark}

 There are now several independent proofs of the original Brascamp-Lieb and Barthe theorems.  Early proofs relied on rearrangement arguments \cite{brascamp1974general, brascamp1976best}, and Lieb appealed to rotational invariance of the extremizers \cite{lieb1990gaussian}.  Barthe came up with a clever optimal transport argument, and simultaneously proved both the forward and reverse inequality \cite{barthe1998reverse}, further establishing equality of best constants.  More recently, Lehec  \cite{lehec2014short} gave a probabilistic proof of both theorems using  a variational representation for functionals due to Bou\'e and Dupuis \cite{boue1998variational}. Semigroup techniques provide yet another avenue of proof; see Bennett, Carbery, Christ and Tao \cite{bennett2008brascamp}, or Carlen, Lieb and Loss \cite{carlen2004sharp}.  Our previous work \cite{liu2018forward} gave an information-theoretic proof of the gaussian saturation part of Theorem \ref{thm:FRBL} (therefore extending to the classical settings as well), by way of a doubling argument similar to that employed by  Geng and Nair in \cite{GengNair} for a different problem.  This doubling argument is similar in spirit to that given by Lieb \cite{lieb1990gaussian}, but it exploited an equivalent entropic representation of the problem.

As far as applications go, it is well-known that the Brascamp-Lieb inequalities imply many other classical inequalities in analysis and geometry, such as H\"older's inequality, Young's inequality, and the Loomis-Whitney inequality.  Likewise, Barthe's inequality contains, for example, the Pr\'ekopa-Leindler  and Brunn-Minkowski inequalities as special cases.  All of these implications and more are described  in Gardner's survey of the Brunn-Minkowski inequality, which places the Brascamp-Lieb, Barthe, and reverse Young inequalities atop a hierarchy of implications  \cite[Figure 1]{gardner2002brunn}, with none of the three evidently implying the others.  In the accompanying discussion, Gardner  asks whether stronger unifying inequalities await discovery; the content of Theorem \ref{thm:FRBL} may be regarded as an affirmative answer.  We have already described how the Brascamp-Lieb and Barthe inequalities may be immediately recognized as special cases of the forward-reverse Brascamp-Lieb inequality.  It turns out that the reverse Young inequality constitutes another instance of the forward-reverse Brascamp-Lieb inequality.  Further examples will be given in Section \ref{sec:connections}.

\begin{example} \label{ex:RevYoung}
Let $0<p,q,r \leq 1$ satisfy $\frac{1}{p} + \frac{1}{q} = 1 + \frac{1}{r}$.  For $\phi,\psi$ non-negative measurable functions on $\R^n$, the reverse Young inequality for convolutions asserts 
\begin{align}
\|\phi * \psi\|_{r}\geq C^n \|\phi\|_p \|\psi\|_q, \label{RYform1}
\end{align}
where $\|h\|_p := \left(\int_{\R^n} |h|^p dx\right)^{1/p}$ for $h: \R^n\to \R$ and $p\in \R$.  The sharp constant is given by $C = C_p C_q/C_r$, with 
$$
C^2_s := \frac{|s|^{1/s}}{|s'|^{1/s'}}
$$
for $1/s + 1/s' = 1$ (i.e., $s$ and $s'$ are H\"older conjugates). 

We may assume $r<1$ henceforth, else if $r=1$, then we must have $p=q=1$, and the claim is trivial.  Under this assumption, it is easily verified using the reverse H\"older inequality and renaming functions, that \eqref{RYform1} is equivalent to
\begin{align}
\iint f_1^{1/p}(x-y) f_2^{1/q}(y) g_1^{1/r'} (x) dx dy  \geq C^n  \left(  \int_{\R^n} f_1    \right)^{1/p}    \left(  \int_{\R^n} f_2     \right)^{1/q}   \left(  \int_{\R^n} g_1     \right)^{1/r'} ,  \label{revYoungEquiv}
\end{align}
holding for  $f_1,f_2,g_1$ non-negative measurable functions on $\R^n$.

To place this into the framework of the forward-reverse Brascamp-Lieb inequality, let $E_1 = E_2 = E^1 = \R^n$ and $E^2=\R^{2n}$, with $c_1 = 1/p$, $c_2 = 1/q$, $d_1 = 1/r-1$ and $d_2 = 1$.  Further, let $\mathbf{B}$ be such that 
$$
\sum_{i=1}^2 c_i B_{i1} x_i = x_1 + x_2; ~~~~ \sum_{i=1}^2 c_i B_{i2} x_i =(x_1,x_2), ~~\forall x_1,x_2\in \R^n.
$$
Then, the hypothesis \eqref{FRBLhypIntro} boils down to
\begin{align}
f_1^{1/p}(x_1) f_2^{1/q}(x_2) \leq g_1^{1/r-1}(x_1+x_2) g_2(x_1,x_2) ~~~~\forall x_1,x_2\in \R^n.
  \label{hypRevYoung}
\end{align}
For arbitrary functions $f_1,f_2,g_1$, the best function $g_2$ can be computed as
$$
g_2(x,y) :=  f_1^{1/p}(x-y) f_2^{1/q}(y) g_1^{1/r'}(x)  ,~~~x,y\in \R^n,
$$
where ``best" is in the sense that the RHS of \eqref{FRBLintro} is minimized subject to \eqref{hypRevYoung}.  On substituting this choice of $g_2$ into \eqref{FRBLintro} and rearranging, we we are left precisely with \eqref{revYoungEquiv}, with best constant necessarily characterized as $C^n = e^{-D(\mathbf{c},\mathbf{d},\mathbf{B}) }$, computed by considering only centered gaussian functions.   
\end{example}

The  relation \eqref{eq:DualIdentity} allows us to easily deduce an interesting ``dual"  to the reverse Young inequality, given in the following example.  Here, we emphasize that the term dual is meant in terms of \eqref{eq:DualIdentity}, which is the same sense in which the Brascamp-Lieb and Barthe inequalities are dual to one another.    It  bears a superficial resemblance to Maurey's property $(\tau)$ \cite{maurey1991} and functional  Santal\'o inequalities (e.g., \cite{artstein2005, lehecSantalo}).  
\begin{example}\label{ex:dualRevYoung}
Let $0<p,q,r \leq 1$   satisfy $\frac{1}{p} + \frac{1}{q} = 1 + \frac{1}{r}$, and let $(\mathbf{c},\mathbf{d},\mathbf{B})$ be the datum of the previous example that yields  the reverse Young inequality \eqref{revYoungEquiv}.  By considering the dual datum $(\mathbf{d},\mathbf{c},\mathbf{B}^*)$ and applying \eqref{eq:DualIdentity}, we  conclude after elementary simplification that
\begin{align}
\int_{\R^{2n}} w \leq K^n  \|f\|_p \|g\|_q \|h\|_{r'} ,\label{dualRevYoung}
\end{align}
for  nonnegative measurable functions $f,g,h,w$ satisfying
\begin{align}
 w(z_1,z_2)\leq \inf_{y\in \R^n }\Big(  f( z_1-y  )g( z_2-y  )  h(y) \Big), ~~z_1,z_2\in \R^n.\label{eq:wdef}
\end{align}
The sharp constant $K$ is given by $K = K_p K_q K_r$, where\footnote{If $s\in \{1,+\infty\}$, $K_s$ is defined by the limit $K_s := \lim_{t\rightarrow s} K_t$ to avoid indeterminate forms in the definition.} 
$$
K_s^2 :=  |s|^{1/s} |s'|^{1/s'}, ~~~0 <s \leq+\infty.
$$
The choice of functions
$$f(x) = e^{|x|^2/p'}, ~~g(x) = e^{ |x|^2/ q'}, ~~h(x) = e^{ |x|^2 /r  } $$
saturate the inequality \eqref{dualRevYoung} when $w$ is taken equal to the infimum in \eqref{eq:wdef}. 
\end{example}
\begin{remark}
 To be more precise, the instances in the above example where $r<1$ follow by \eqref{eq:DualIdentity}, and the exceptional case $r=1$ ($r'=-\infty$) can be easily checked directly. %
\end{remark}

\begin{remark}
In analogy to Example \ref{ex:dualRevYoung}, one can derive  a ``dual" Young inequality.  It formally reverses the inequality of the previous example:  If $p,q,r\geq 1$ satisfy $\frac{1}{p} + \frac{1}{q} = 1 + \frac{1}{r}$, then 
\begin{align}
\int_{\R^{2n}} w \geq K^n \|f\|_p \|g\|_q  \|h\|_{r'}\label{dualYoung}
\end{align}
for  nonnegative measurable functions $f,g,h,w$ satisfying
\begin{align}
 w(z_1,z_2)\geq \sup_{y\in \R^n }\Big(  f( z_1-y  )g( z_2-y  ) h(y) \Big), ~~z_1,z_2\in \R^n.\label{eq:w2def}
\end{align}
Since the standard Young  inequality is a special case of the Brascamp-Lieb inequality, its dual \eqref{dualYoung} is a special case of the Barthe inequality, and should therefore be regarded as already known.  In contrast, we do not know whether   \eqref{dualRevYoung} (or  equivalent) has appeared previously in the literature. 
\end{remark}

To close this section, let us remark briefly on the chief difficulty encountered in proving Theorem \ref{thm:FRBL} compared to the special cases corresponding to the Brascamp-Lieb and Barthe inequalities.  In the case of the direct Brascamp-Lieb inequalities, the function $f_1$ can be explicitly computed in terms of the $(g_j)_{1\leq j\leq m}$ (specifically, $f_1 = \prod_{j=1}^m (g_j\circ B_j)^{d_j}$).  In the reverse case, the function $g_1$ can be computed explicitly in terms of the $(f_i)_{1\leq i\leq k}$.  Such a simplification is not typically possible in the more general forward-reverse inequalities; this leads us to establish a rather  subtle duality principle, to be made precise in Theorem \ref{thm:equivFormsDg} (see Remark \ref{rmk:differenceFromClassical}).  Once this duality principle is established,  the structure of gaussian extremizers can be distilled, and techniques previously developed for proving the Brascamp-Lieb and Barthe inequalities can be successfully adapted to forward-reverse setting.

\subsection{Finiteness and gaussian-extremizability}
A motivation of the present paper is to better understand the structural properties of the gaussian extremizers in Theorem \ref{thm:FRBL} (when they exist), and to establish necessary and sufficient conditions for finiteness of   $D(\mathbf{c},\mathbf{d},\mathbf{B})$.  To this end, we give a new proof of Theorem \ref{thm:FRBL}, which combines ideas from Lehec's probabilistic proof of the Brascamp-Lieb and Barthe inequalities \cite{lehec2014short}, the structural viewpoint on the forward Brascamp-Lieb inequality developed by Bennett, Carbery, Christ and Tao \cite{bennett2008brascamp}, and the entropic duality of the forward-reverse Brascamp-Lieb inequality in \cite{liu2018forward}.   The detailed structural results, for example, allow us to easily identify ``geometric" instances of the forward-reverse Brascamp-Lieb inequality, which may be particularly useful in applications (see, e.g., Section \ref{sec:gaussianKernels}).%

Let us now  make precise the notions of gaussian-extremizability and extremizers that have been  alluded to above.  We will need some more notation, which will prevail throughout.  For a Euclidean space $E$, we let $S(E)$ denote the set of self-adjoint linear operators on $E$, and $S^+(E)$ denote the set of self-adjoint positive-definite linear operators on $E$.  That is, $A\in S^+(E)$ if $A\in S(E)$  and it further holds that $\langle A x, x\rangle>0 $ for all nonzero $x\in E$.   If $A\in S(E)$ and $\langle A x, x\rangle\geq0 $ for all  $x\in E$, we say that $A$ is positive-semidefinite.  For $A,B \in S(E)$, we write $A\geq B$ if $A-B$ is  positive-semidefinite.    Finally, for positive-semidefinite $A$, we let $A^{1/2}$ denote the unique positive-semidefinite $M$ such that $A= M^2$. 

A centered gaussian function (or kernel) $g: E\to \R^+$ is a function of the form 
$$
g(x) =   \exp\left(-\frac{1}{2}\langle A^{-1} x, x \rangle\right),  ~~~A \in S^+(E),
$$
where $A$ is said to be the \emph{covariance} of the gaussian kernel $g$.  We remark that a centered gaussian random vector on $E$ with covariance $A$ has density (with respect to Lebesgue measure on $E$) proportional to $g$.

Restricting attention to centered gaussian functions 
\begin{align*}
f_i &: x\in E_i  \longmapsto  \exp(-\langle V_i x, x \rangle),  ~~~V_i \in S^+(E_i), 1\leq i\leq k,\\
g_j &: x\in E^j \longmapsto  \exp(-\langle U_j x, x \rangle),  ~~~U_j \in S^+(E^j),1\leq j\leq m,
\end{align*}
and defining  
$$
\Lc := \bigoplus_{i=1}^k c_i \operatorname{id}_{E_i},
$$
we see that the hypothesis \eqref{FRBLhypIntro} boils down to 
\begin{align}
\sum_{j=1}^m d_j \langle U_j B_{j} \Lc x  , B_{j} \Lc x \rangle  \leq \sum_{i=1}^k c_i \langle V_i \pi_{E_i}(x), \pi_{E_i}(x) \rangle   \hspace{1cm}\forall x \in  E_0. \label{Qconstraint}
\end{align}
Additionally, using \eqref{eq:consistency}, we may compute
$$
\frac{\prod_{i=1}^k \left( \int_{E_i} f_i  \right)^{c_i} }{ \prod_{j=1}^m \left( \int_{E^j} g_j  \right)^{d_j} } = \frac{\prod_{j=1}^m \det(U_j)^{d_j/2}}{\prod_{i=1}^k \det(V_i)^{c_i/2}}.
$$
Collecting the above and comparing to \eqref{FRBLintro},  this motivates definition of the quantity
\begin{align}
D_g(\mathbf{c},\mathbf{d},\mathbf{B}) := \frac{1}{2}\sup \left( \sum_{j=1}^m d_j \log \det U_j - \sum_{i=1}^k c_i \log \det V_i\right), \label{defDg}
\end{align}
where the supremum is over all $V_i\in S^+(E_i),1\leq i\leq k$ and $U_j\in S^+(E^j), 1\leq j\leq m$ satisfying \eqref{Qconstraint}.  In words, $D_g(\mathbf{c},\mathbf{d},\mathbf{B})$ is the smallest possible constant $D$ in \eqref{FRBLintro}, holding for all gaussian kernels satisfying the constraints \eqref{FRBLhypIntro}.  
By definition, $D_g(\mathbf{c},\mathbf{d},\mathbf{B})\leq D(\mathbf{c},\mathbf{d},\mathbf{B})$.  The first part of Theorem \ref{thm:FRBL} asserts that this is always met with equality. 

\begin{remark}
Even if \eqref{eq:consistency} does not hold, it remains true that $D_g(\mathbf{c},\mathbf{d},\mathbf{B})$ as defined above will equal the smallest possible constant $D$ in \eqref{FRBLintro}, holding for all gaussian kernels satisfying the constraints \eqref{FRBLhypIntro}.  Indeed, scaling the $(V_i)_{1\leq i\leq k}$ and $(U_j)_{1\leq j \leq m}$ by a common factor shows $D_g(\mathbf{c},\mathbf{d},\mathbf{B})=+\infty$, while dilating functions by a common factor will show that the best constant $D$ in \eqref{FRBLintro} will also be equal to $+\infty$. 
\end{remark}

\begin{definition}
A datum $(\mathbf{c},\mathbf{d},\mathbf{B})$ is \emph{gaussian-extremizable} if the supremum in \eqref{defDg} is attained and is finite for some $V_i\in S^+(E_i),1\leq i\leq k$ and $U_j\in S^+(E^j), 1\leq j\leq m$ satisfying \eqref{Qconstraint}.  Any such operators are said to be \emph{gaussian extremizers}. 
\end{definition}

The constraint \eqref{Qconstraint} is a bit cumbersome to write out.  So, henceforth, we adopt some notation to make statements more compact; for given $V_i\in S^+(E_i),1\leq i\leq k$ and $U_j\in S^+(E^j), 1\leq j\leq m$,  we define  $V_{\mathbf{c}} : E_0 \to E_0$ and $U_{\mathbf{d}} : \bigoplus_{j=1}^m E^j \to \bigoplus_{j=1}^m E^j $ according to 
$$
V_{\mathbf{c}} := \bigoplus_{i=1}^k c_i V_i; \hspace{1cm}U_{\mathbf{d}} := \bigoplus_{j=1}^m d_j U_j.
$$
Note that this does not cause any ambiguity since $V_i\in S^+(E_i),1\leq i\leq k$ defines $V_{\mathbf{c}}$, and vice versa.  Similarly for $U_{\mathbf{d}}$.  The constraints \eqref{Qconstraint} may now be concisely written as the operator inequality
\begin{align}
\Lc B^* U_{\mathbf{d}} B \Lc  \leq V_{\mathbf{c}}, \label{operatorInequality}
\end{align}
where $B : E_0 \to \bigoplus_{j=1}^k E^j$ is the linear operator defined according to
$$
B : x \longmapsto B_1 x +  \cdots +  B_m x; ~~~~x \in E_0.
$$

We introduce one last piece of notation before our characterization of gaussian-extremizability: 
\begin{definition}
For $A_i \in S^{+}(E_i)$, $i=1,\dots, k$, we let $\Pi(A_1, \dots, A_k)$ denote the set of positive-semidefinite $A\in S(E_0)$ satisfying 
$$
\langle A x_i,x_i\rangle   = \langle A_i x_i, x_i\rangle  \mbox{~~for all $x_i\in E_i$, } 1\leq i\leq k.
$$
\end{definition}
\begin{remark}\label{rmk:PiasCouplings}
The above definition has a natural interpretation in terms of couplings.  Consider a gaussian random vector $X_i$ taking values in $E_i$, with covariance $A_i$, $1\leq i \leq k$.  If a jointly gaussian coupling of $(X_i)_{1\leq i\leq k}$ has covariance $A$, then  $A\in \Pi(A_1, \dots, A_k)$.  Conversely, each $A\in \Pi(A_1, \dots, A_k)$ corresponds to the covariance of a jointly gaussian coupling of $(X_i)_{1\leq i\leq k}$. 
\end{remark}

\begin{theorem}[Structure of gaussian extremizers]\label{thm:FRBLgaussianExtiff}
A datum $(\mathbf{c},\mathbf{d},\mathbf{B})$ is gaussian-extremizable if and only if \eqref{eq:consistency} holds and there are  $V_i\in S^+(E_i), 1\leq i\leq  k$ and  $\Pi  \in \Pi(V^{-1}_1,\dots, V^{-1}_k)$ satisfying 
\begin{align}
\sum_{j=1}^m d_j   \Lc B_j^*\left( B_j \Lc \Pi \Lc B_j^*\right)^{-1} B_j\Lc   \leq  V_{\mathbf{c}}   . \label{optimalityCondition}
\end{align}
Moreover, any such $(V_i)_{1\leq i\leq k }$ together with $U_j = ( B_j \Lc \Pi \Lc B_j^* )^{-1}$, $1\leq j \leq m$, are gaussian extremizers.
\end{theorem}
\begin{remark}\label{rmk:restrictSubspace}
Implicit in \eqref{optimalityCondition} is the assertion that the stated inverses exist;  it is therefore necessary for each $(B_j)_{1\leq j\leq m}$ to be surjective in order for the datum  $(\mathbf{c},\mathbf{d},\mathbf{B})$ to be gaussian extremizable.   Moreover, after  left- and right-multiplying  both sides of \eqref{optimalityCondition} by $\Pi^{1/2}$, the traces of the respective sides will be equal by \eqref{eq:consistency}.  Hence, \eqref{optimalityCondition} is met with equality on  restriction to the subspace $ \Pi E_0$.  See also Remark \ref{rmk:restrictionOfGaussExtremizableCondtions} and Proposition \ref{prop:EquivOptimalConds} for more   along these lines.
\end{remark}
\begin{remark}
In view of the previous remark, for the classical setting of $k=c_1=1$, gaussian-extremizability reduces to \eqref{eq:consistency} and  existence of $V \in S^+(E_1)$ such that 
\begin{align*}
\sum_{j=1}^m d_j     B_j^*\left( B_j   V^{-1}   B_j^*\right)^{-1} B_j    =  V . 
\end{align*}
This  has been repeatedly observed in previous proofs of the Brascamp-Lieb and Barthe inequalities.  
\end{remark}

The  \emph{geometric Brascamp-Lieb inequalities}  proposed by Ball \cite{ball1989volumes} and later generalized by Barthe \cite{barthe1998reverse} are a special case of the (forward) Brascamp-Lieb inequalities for which the linear maps are isometric and $D(\mathbf{c},\mathbf{d},\mathbf{B}) =0$.  The class of geometric Brascamp-Lieb inequalities are  useful in applications (see, e.g., the volume ratio inequalities due to K.~Ball), and are formally equivalent to the class of gaussian-extremizable Brascamp-Lieb inequalities \cite[Proposition 3.6]{bennett2008brascamp}. The following is a generalization to the forward-reverse setting:
\begin{corollary}[Geometric forward-reverse Brascamp-Lieb inequality (I)]  \label{cor:geometric}  Assume \eqref{eq:consistency} holds, and let  linear maps $Q_j : E_0 \to E^j$  and $\Sigma \in \Pi(\operatorname{id}_{E_1}, \dots, \operatorname{id}_{E_k})$  satisfy 
$$Q_j \Sigma Q_j^* = \operatorname{id}_{E^j} ~\mbox{for each~}1\leq j \leq m,  \hspace{7mm} \mbox{and} \hspace{7mm} \sum_{j=1}^m d_j Q_j^* Q_j \leq \Lc.$$ 
If measurable functions $f_i : E_i \to \R^+$,  $1\leq i \leq k$ and $g_j : E^j \to \R^+$,  $1\leq j\leq m$  satisfy 
\begin{align}
\prod_{i=1}^k f_i^{c_i}(\pi_{E_i}(x) ) \leq \prod_{j=1}^m g_j^{d_j}\left( Q_j x \right)\hspace{1cm}\forall x \in E_0,\label{FRBLhypGeom}
\end{align}
then 
\begin{align}
\prod_{i=1}^k \left( \int_{E_i} f_i  \right)^{c_i} \leq  \prod_{j=1}^m \left( \int_{E^j} g_j  \right)^{d_j}.\label{FRBLGeom}
\end{align}
\end{corollary}
\begin{proof} By defining the maps $B_j$ (and therefore $\mathbf{B}$) via $B_j \Lc = Q_j$, the hypothesis \eqref{FRBLhypGeom} coincides with \eqref{FRBLhypIntro}. For the corresponding datum $(\mathbf{c},\mathbf{d},\mathbf{B})$, we see that $V_i = \operatorname{id}_{E_i}$, $1\leq i\leq k$, and $\Pi = \Sigma$ satisfy \eqref{optimalityCondition} by using the assumptions on the $(Q_j)_{1\leq j\leq m}$.  Now, it is a matter of plugging in the asserted extremizers  in Theorem \ref{thm:FRBLgaussianExtiff} into the definition of $D_g(\mathbf{c},\mathbf{d},\mathbf{B})$ to see that $D_g(\mathbf{c},\mathbf{d},\mathbf{B}) = 0$, and therefore $D(\mathbf{c},\mathbf{d},\mathbf{B}) =0$ by Theorem \ref{thm:FRBL}. This gives \eqref{FRBLGeom} by definition. \end{proof}

\begin{remark}
In general, $\Sigma$ in Corollary \ref{cor:geometric} does not need to be of full rank.  An illustrative example is Barthe's inequality, in which $m=d_1=1$, $E^1=\R^n$ and $E_i =\R$ for each $1\leq i \leq k$, with $k\geq n$.  In this setting $Q_1$, viewed as a matrix, has columns $c_i q_i \in \R^{n}$,  where  $|q_i|=1$, $1\leq i\leq k$.  The reader can check that for the choice $\Sigma = \Lc^{-1}Q_1^* Q_1 \Lc^{-1} \in S^+(\R^k)$, the assumptions of the corollary are equivalent to Barthe's frame condition 
$$
|q_i|=1, ~1\leq i\leq k; \mbox{~~~and~~~}\sum_{i=1}^k c_i q_i \otimes q_i = I_{k\times k}.
$$
Note that this  $\Sigma$ has rank at most $n\leq k$, so can be rank-deficient.  
\end{remark}

Although it is a special case, a more symmetric formulation of the geometric forward-reverse  Brascamp-Lieb inequality  may be stated as follows, and explains the role of $\Sigma$ in the previous as a transformation of coordinates when it is assumed to be of full rank.  By specializing to $k=c_1=1$, the geometric case  of the Brascamp-Lieb inequality is easily recognized.  It will also be useful for the applications  in Section \ref{sec:gaussianKernels}.
\begin{corollary}[Geometric forward-reverse Brascamp-Lieb inequality (II)]  \label{cor:geometric2}  
Let linear maps $U_i : E_0\to E_i$ and $V_j : E_0 \to E^j$  satisfy   $U_i U_i^* = \operatorname{id}_{E_i}$ and  $V_j V_j^* = \operatorname{id}_{E^j}$, for all $1\leq i \leq k$ and $1\leq j\leq m$.  Assume further that the map
$$
x \in E_0 \longmapsto \sum_{i=1}^k U_i x \in E_0
$$
is a bijection, and that the following frame condition holds
$$
\sum_{i=1}^k c_i U_i^* U_i = \sum_{j=1}^m d_j V_j^* V_j. 
$$
If measurable functions $f_i : E_i \to \R^+$,  $1\leq i \leq k$ and $g_j : E^j \to \R^+$,  $1\leq j\leq m$  satisfy 
\begin{align}
\prod_{i=1}^k f_i^{c_i}(U_i x ) \leq \prod_{j=1}^m g_j^{d_j}\left( V_j x \right)\hspace{1cm}\forall x \in E_0,\label{FRBLhypGeom2}
\end{align}
then 
\begin{align}
\prod_{i=1}^k \left( \int_{E_i} f_i  \right)^{c_i} \leq  \prod_{j=1}^m \left( \int_{E^j} g_j  \right)^{d_j}.\label{FRBLGeom2}
\end{align}
\end{corollary}
\begin{remark}\label{rmk:equivSurj}
The map $ \sum_{i=1}^k U_i : E_0\to E_0$ being a bijection is equivalent to $\sum_{i=1}^k c_i U_i^* U_i >0$.  
\end{remark}
\begin{proof}
By taking traces, the frame condition ensures that \eqref{eq:consistency} holds.  Next, view $U_i$ as a linear map from $E_0$ into itself, so that $\ker(U_i^*)=E_i^{\perp}$ for each $1\leq i\leq k$.  Since $\sum_{i=1}^k U_i$ is a bijection, it is invertible, and therefore we are justified in defining  $Q_j := V_j (\sum_{i=1}^k U_i)^{-1}$ and $\Sigma := \left( \sum_{i=1}^k U_i \right) \left( \sum_{i=1}^k U_i \right)^*$.  Evidently, $\Sigma$ is positive-definite, and for $x_i \in E_i$, we have $\langle \Sigma x_i,x_i\rangle = \langle x_i,x_i\rangle$ using the assumption  $U_i U_i^* = \operatorname{id}_{E_i}$ and identification of $\ker(U_i^*)=E_i^{\perp}$.  Therefore, $\Sigma \in \Pi(\operatorname{id}_{E_1}, \dots, \operatorname{id}_{E_k})$.  Furthermore, it follows from definitions that $Q_j \Sigma Q_j^* = \operatorname{id}_{E^j}$ for each $1\leq j \leq m$.  Using again the fact that $\sum_{i=1}^k U_i$ is a bijection, we find that  \eqref{FRBLhypGeom2} and \eqref{FRBLhypGeom} are equivalent by a change of variables.   Thus, the hypotheses of Corollary \ref{cor:geometric} are fulfilled, and the conclusion follows. 
\end{proof}

The first formulation of the geometric forward-reverse  Brascamp-Lieb inequality motivates the following definitions:
\begin{definition}
A datum  $(\mathbf{c},\mathbf{d},\mathbf{B})$ is said to be \emph{geometric} if \eqref{eq:consistency} holds,  and the maps $(Q_j)_{1\leq j\leq m}$ defined via $Q_j := B_j \Lc$ satisfy  $\sum_{j=1}^m d_j Q_j^* Q_j \leq \Lc$ and $Q_j \Sigma Q_j^* = \operatorname{id}_{E^j}$, $1\leq j \leq m$,  for some $\Sigma \in \Pi(\operatorname{id}_{E_1}, \dots, \operatorname{id}_{E_k})$.
\end{definition}

\begin{definition}
Data $(\mathbf{c},\mathbf{d},\mathbf{B})$ and $(\mathbf{c}',\mathbf{d}',\mathbf{B}')$ are said to be equivalent if $\mathbf{c}=\mathbf{c}'$,  $\mathbf{d}=\mathbf{d}'$, and there exist  invertible linear transformations $C: E_0\to E_0$ and $C_j: E^j\to E^j$ such that $B_j' = C_j^{-1} B_j C^{-1}$ for each $1\leq j\leq k$.  
\end{definition}

The following characterization of gaussian extremizability extends  \cite[Proposition 3.6]{bennett2008brascamp} to the forward-reverse setting. 
\begin{theorem}
A datum $(\mathbf{c},\mathbf{d},\mathbf{B})$ is gaussian-extremizable if and only if it is equivalent to a geometric datum $(\mathbf{c},\mathbf{d},\mathbf{B}')$. 
\end{theorem}
\begin{proof}
Suppose $(\mathbf{c},\mathbf{d},\mathbf{B})$ is gaussian-extremizable.  This is equivalent to the existence of $V_i\in S^+(E_i), 1\leq i\leq  k$ and  $\Pi  \in \Pi(V^{-1}_1,\dots, V^{-1}_k)$ satisfying \eqref{optimalityCondition}.  Define $V := \Lc^{-1}V_{\mathbf{c}} = \bigoplus_{i=1}^k V_i$, and note that $\Sigma:= V^{1/2} \Pi V^{1/2} \in  \Pi(\operatorname{id}_{E_1}, \dots, \operatorname{id}_{E_k})$, and moreover, that $V^{-1/2}$ commutes with $\Lc$.  Define $B'_j := C_j^{-1} B_j C^{-1}$ via the  transformations $C_j := (B_j \Lc \Pi \Lc B_j^*)^{1/2}$ and $C:= V^{1/2}$.  Then,  \eqref{optimalityCondition} is expressed in terms of $\mathbf{B}':=(B'_j)_{1\leq j \leq m}$ and $\Sigma$ as
$$
\sum_{j=1}^m d_j \Lc {B'}^{*}_j  \left(  B'_j \Lc \Sigma \Lc {B'}^{*}_j  \right)^{-1} {B'}_j \Lc \leq \Lc. 
$$
In particular, this easily implies $(\mathbf{c},\mathbf{d},\mathbf{B}')$ is geometric since $B'_j \Lc \Sigma \Lc {B'}^{*}_j = \operatorname{id}_{E^j}$ by construction.   Moreover, $(\mathbf{c},\mathbf{d},\mathbf{B})$ and $(\mathbf{c},\mathbf{d},\mathbf{B}')$ are equivalent by definition.  If $(\mathbf{c},\mathbf{d},\mathbf{B})$ is equivalent to a geometric datum, then the argument can be  reversed to conclude gaussian-extremizability via Theorem \ref{thm:FRBLgaussianExtiff}. 
\end{proof}

To state our main result on conditions for finiteness of $D_g(\mathbf{c},\mathbf{d},\mathbf{B})$ and gaussian-extremizability, we define product-form subspaces, followed by two definitions analogous to those  in \cite{bennett2008brascamp}.
\begin{definition}
A subspace $T$ is said to be of \emph{product-form} if 
$$
T = \bigoplus_{i=1}^k \pi_{E_i}(T)
$$
\end{definition}

\begin{definition}A \emph{critical subspace} for  $(\mathbf{c},\mathbf{d},\mathbf{B})$ is a non-zero proper subspace $T\subset E_0$ of product-form satisfying
$$
\sum_{i=1}^k c_i \dim(\pi_{E_i}T) = \sum_{j=1}^m d_j \dim(B_j T) .
$$
\end{definition}
\begin{definition}
The datum  $(\mathbf{c},\mathbf{d},\mathbf{B})$ is \emph{simple} if it has no critical subspaces. 
\end{definition}
We remark that the definition of criticality in \cite{bennett2008brascamp} does not include the restriction to product-form subspaces.  However, the two definitions are still consistent, because in their setting $E_0=E_1$, so that any subspace is trivially of product-form.

Our final  main result generalizes  \cite[Theorem 1.13]{bennett2008brascamp} to the setting of the forward-reverse Brascamp-Lieb inequality:
\begin{theorem}[Conditions for finiteness and gaussian-extremizability]\label{thm:necSuffCondDgFinite}
For a datum $(\mathbf{c},\mathbf{d},\mathbf{B})$, the quantity $D_g(\mathbf{c},\mathbf{d},\mathbf{B})$ is finite if and only if we have the scaling condition
\begin{align}
\sum_{i=1}^k c_i \dim(E_i) = \sum_{j=1}^m d_j \dim(E^j) \label{CONDITIONscaling}
\end{align}
and the dimension condition
\begin{align}
\sum_{i=1}^k c_i \dim(\pi_{E_i}T) \leq  \sum_{j=1}^m d_j \dim(B_j T) ~~\mbox{for all product-form subspaces $T\subseteq E_0$.} \label{CONDITIONdimension}
\end{align}
In particular, these conditions imply that each $(B_j)_{1\leq j\leq m}$ must be surjective. 
Moreover, if $(\mathbf{c},\mathbf{d},\mathbf{B})$ is simple, then it is gaussian-extremizable. 
\end{theorem}
\begin{remark}
The special case where $k=c_1=1$ reduces to \cite[Theorem 1.13]{bennett2008brascamp}. 
\end{remark}
\begin{remark}
By Theorem \ref{thm:FRBL}, finiteness of $D_g(\mathbf{c},\mathbf{d},\mathbf{B})$ is equivalent to finiteness of $D_g(\mathbf{d},\mathbf{c},\mathbf{B}^*)$.  As expected, it can also be verified directly that the conditions of Theorem \ref{thm:necSuffCondDgFinite} are  invariant under considering the dual datum $(\mathbf{c},\mathbf{d},\mathbf{B})\to (\mathbf{d},\mathbf{c},\mathbf{B}^*)$.  Indeed, let $E^0 := \bigoplus_{j=1}^m E^j$, and define $B^i : E^0 \to E_i$ via the map
$$
B^i y := \sum_{j=1}^m B^*_{ij} \pi_{E^j}(y), ~~~y\in E^0. 
$$ 
For any $W_j \subseteq E^j$, $1\leq j\leq m$, define the product-form subspace $W = \bigoplus_{j=1}^m W_j\subseteq E^0$, and consider $T = \bigoplus_{i=1}^k T_i$, with $T_i := E_i / B^i W$.  By the inclusion $ \bigoplus_{i=1}^k B^i W\supseteq B^* W \supseteq B_j^* W_j$, we have
$$
B_j T  \subseteq B_j ( E_0 / (B^* W)   ) \subseteq B_j ( E_0 / (B_j^* W_j  )   ) \subseteq E^j /W_j,
$$
where the final equality follows again by set inclusion and the observation $B_j^* W_j\subseteq \ker(B_j)^{\perp}$.  Hence, \eqref{CONDITIONdimension} implies
$$
\sum_{i=1}^k c_i \dim(E_i)  - \sum_{i=1}^k c_i \dim(B^i W) \leq  \sum_{j=1}^m d_j \dim(E^j) -  \sum_{j=1}^m d_j \dim(W_j).
$$
Cancelling terms using \eqref{CONDITIONscaling} (which is trivially invariant to considering dual datums), we are left with the desired dual counterpart to \eqref{CONDITIONdimension}.
\end{remark}

\subsection{Outline of the paper}
Section \ref{sec:GaussianExtremizableCase} of this paper proves Theorem \ref{thm:FRBL} under the assumption that the datum  is gaussian-extremizable.  This relies on establishing the structure of gaussian extremizers given in Theorem \ref{thm:FRBLgaussianExtiff}, and then adapting Lehec's stochastic proof of the Brascamp-Lieb and Barthe inequalities to exploit this information. 

It turns out that the analysis of the gaussian-extremizable case more or less suffices to prove Theorem \ref{thm:FRBL} in its full generality.  To do this, we develop a machinery for iteratively decomposing a datum that is not gaussian-extremizable.  This is the general focus of Section \ref{sec:Decomposability}, which parallels the  development of Bennett, Carbery, Christ and Tao for the  special case of the direct Brascamp-Lieb inequality \cite{bennett2008brascamp}.  The conditions for finiteness and gaussian-extremizability articulated in Theorem \ref{thm:necSuffCondDgFinite} are a product of these arguments.

Connections between the forward-reverse Brascamp-Lieb inequality and other results in the literature are detailed in Section \ref{sec:connections}.

\subsection*{Acknowledgment}
The first author is grateful for discussions with Franck Barthe, Michael Christ, Varun Jog, Joseph Lehec and Pawe\l{} Wolff. In particular, he thanks Franck Barthe for drawing our attention to references \cite{barthe2014positivity, barthe2018positive}, and Pawe\l{} Wolff  for explaining the connections to the forward-reverse Brascamp-Lieb inequality.

\section{The Gaussian-Extremizable Case}\label{sec:GaussianExtremizableCase}

The goal of this section is to establish Theorem \ref{thm:FRBL} under the assumption of gaussian-extremizability.  Specifically, we aim to prove the following two results:
\begin{theorem}\label{thm:FRBLunderGaussExt}
If a datum $(\mathbf{c},\mathbf{d},\mathbf{B})$ is gaussian-extremizable, then $D(\mathbf{c},\mathbf{d},\mathbf{B})=D_g(\mathbf{c},\mathbf{d},\mathbf{B})$.
\end{theorem}
\begin{theorem}\label{thm:FRBLdualityConstants}
If a datum $(\mathbf{c},\mathbf{d},\mathbf{B})$ is gaussian-extremizable, then so is $(\mathbf{d},\mathbf{c},\mathbf{B}^*)$.  Moreover, 
$$D_g(\mathbf{c},\mathbf{d},\mathbf{B})=D_g(\mathbf{d},\mathbf{c},\mathbf{B}^*),$$
regardless of whether the data are gaussian-extremizable. 
\end{theorem}
\begin{remark}
If $D_g(\mathbf{c},\mathbf{d},\mathbf{B})=+\infty$, then the datum $(\mathbf{c},\mathbf{d},\mathbf{B})$ does not satisfy the definition of gaussian-extremizability.  However, we will clearly have $D(\mathbf{c},\mathbf{d},\mathbf{B})=D_g(\mathbf{c},\mathbf{d},\mathbf{B})$ in this case also due to the general relation $D(\mathbf{c},\mathbf{d},\mathbf{B})\geq D_g(\mathbf{c},\mathbf{d},\mathbf{B})$.
\end{remark}
\begin{remark}
The combination of the above   implies the assertion of Theorem \ref{thm:FRBL} if $(i)$ $(\mathbf{c},\mathbf{d},\mathbf{B})$ is gaussian-extremizable; or $(ii)$ $D_g(\mathbf{c},\mathbf{d},\mathbf{B})=+\infty$. 
\end{remark}
\subsection{Proof of Theorems \ref{thm:FRBLgaussianExtiff} and \ref{thm:FRBLunderGaussExt}}\label{sec:Lehec-Style}
Let us assume the following preliminary version of Theorem \ref{thm:FRBLgaussianExtiff}, the proof of which is deferred to Section \ref{sec:ProofofOptimalityProp}.
\begin{proposition}\label{prop:OptimCond}
If a datum $(\mathbf{c},\mathbf{d},\mathbf{B})$ is gaussian-extremizable, then \eqref{eq:consistency} holds and there are  $V_i\in S^+(E_i)$, $1\leq i \leq  k$ and  $\Pi  \in \Pi(V^{-1}_1,\dots, V^{-1}_k)$ such that 
\begin{align}
\sum_{j=1}^m d_j   \Lc B_j^*\left( B_j \Lc \Pi \Lc B_j^*\right)^{-1} B_j\Lc   \leq  V_{\mathbf{c}}   .   \label{propOptim}
\end{align}
\end{proposition}

We now describe the variational representation for gaussian integrals due to Bou\'e and Dupuis \cite{boue1998variational} (see also Borell \cite{borell2000diffusion}, and historical remarks by Lehec \cite{lehec2013representation}).  For a given time horizon $T$ and a Euclidean space $E$, a Brownian motion $(W_t)_{0\leq t \leq T}$ (starting from 0) taking values in $E$ is said to have covariance $K\in S^{+}(E)$ if $\Cov(W_1)=K$.  Let $\mathbb{H}(E,K)$ denote the Hilbert space of all absolutely continuous paths $u: [0,T]\to E$ starting from $0$, equipped with norm 
$$
\|u \|^2_{\mathbb{H}(E,K)} := \int_0^T \langle K^{-1} \dot{u}_s,\dot{u}_s\rangle ds. 
$$
A \emph{drift} $U$ is any process adapted to the Brownian filtration which has sample paths belonging to $\mathbb{H}(E,K)$ almost surely. 

\begin{proposition}\label{prop:GaussRep}
Let $g: E \to \R$ be measurable and bounded from below, and let $(W_t)_{0\leq t \leq T}$ be a Brownian motion with covariance $K\in S^{+}(E)$.  It holds that
 $$
\log \left( \EE e^{g(W_T)} \right) = \sup  \EE\left[ g(W_T+U_T) - \frac{1}{2}\|U\|^2_{\mathbb{H}(E,K)} \right],  
$$
where the supremum is taken over all drifts $U$.
\end{proposition}

We now prove Theorems \ref{thm:FRBLgaussianExtiff} and \ref{thm:FRBLunderGaussExt}, on the basis of Proposition  \ref{prop:OptimCond}.  The argument is an adaptation of Lehec's proof of the Brascamp-Lieb and Barthe inequalities \cite{lehec2014short}, with the main difference being the incorporation of the optimality conditions of Proposition \ref{prop:OptimCond}. 
\begin{proof}[Proof of Theorems \ref{thm:FRBLgaussianExtiff} and \ref{thm:FRBLunderGaussExt}]
Assume $(f_i)_{1\leq i\leq k}$ and $(g_j)_{1\leq j\leq m}$ satisfy \eqref{FRBLhypIntro}.  For purposes of proving the theorem, we may assume that each $f_i$ is supported on some compact $K_i\subset E_i$, and is bounded from above, say by $M$.  We may also assume that each $g_j$ is bounded from below, say by $M^{-1}$ on the compact set $\sum_{i=1}^{k}c_k B_{ij} K_i \subset E^j$.  The general result will follow by dominated convergence.  As a result, it is easy to see that we may now assume each $g_j$ is bounded from above by some $M' = M'(M,\mathbf{c},\mathbf{d})$, still preserving \eqref{FRBLhypIntro}. Indeed, this modification can only reduce the product in the RHS of \eqref{FRBLintro}, making our task more difficult.

With the above assumptions, fix any $\delta>0$ and introduce the auxiliary functions
$$
u_i = \log(f_i + \delta)
$$
which are of course bounded from below.  Using the assumption that the $f_i$'s and $g_j$'s are uniformly bounded and the hypothesis \eqref{FRBLhypIntro}, there are constants $C,c$ (depending on $M,M', \mathbf{c},\mathbf{d}$, but not  on $\delta$) such that taking 
$$
v_j =\log(g_j + C\delta^{c})
$$
we will have
$$
\sum_{i=1}^k c_i u_i(\pi_{E_i}(x)) \leq \sum_{j=1}^m d_j v_j (B_j \Lc x)~~~~~\forall x\in E_0.
$$
Moreover, the $v_j$'s are also bounded from below.

Using the assumption of gaussian-extremizability, we invoke Proposition \ref{prop:OptimCond}  to select $V_i \in S^+(E_i)$, $i=1, \dots, k$ and $\Pi \in \Pi(V^{-1}_1,\dots, V^{-1}_k)$ satisfying \eqref{propOptim}.   %
In particular, this implies
\begin{align}
\sum_{j=1}^m d_j \|B_j \Lc u \|^2_{ \mathbb{H}(E^j, B_j \Lc \Pi \Lc B_j^*)}  \leq \sum_{i=1}^k c_i \|\pi_{E_i}(u)\|^2_{\mathbb{H}(E_i,V_i^{-1})}  \label{Hnorms}
\end{align}
for any absolutely continuous path $u: [0,T]\to E_0$.  

Now, let $(W_t)_{0\leq t \leq T}$ be a Brownian motion taking values in $E_0$ with covariance $\Pi$.  For each $i=1,\dots, k$, define $W^i_t = \pi_{E_i}(W_t)$, which is a Brownian motion on $E_i$ with covariance $V^{-1}_i$. 

Fix $\epsilon>0$, and for each $i=1,\dots, k$, let $U^i$ be an $E_i$-valued drift belonging to $\mathbb{H}(E_i,V_i^{-1})$ such that
$$
\log \EE e^{u_i(W_T^i)} - c_i^{-1}\epsilon \leq  \EE\left[   u_i(W^i_T+U^i_T) - \frac{1}{2}\|U^i\|^2_{\mathbb{H}(E_i,V_i^{-1})} \right],
$$
the existence of which follows from Proposition \ref{prop:GaussRep}. Define the $E_0$-valued process $U =\sum_{i=1}^k U^k$, and note that $B_j \Lc U$ is a drift  belonging to $\mathbb{H}(E^j, B_j \Lc \Pi \Lc B_j^*)$ by \eqref{Hnorms}.

Hence, the above estimates and another application of Proposition \ref{prop:GaussRep} give
\begin{align*}
&\sum_{i=1}^k c_i \log \EE e^{u_i(W_T^i)} - \epsilon \\
&\leq \EE\left[   \sum_{i=1}^k c_i u_i(W^i_T+U^i_T) -  \frac{1}{2} \sum_{i=1}^k c_i \|U^i\|^2_{\mathbb{H}(E_i,V_i^{-1})} \right]\\
&\leq \EE\left[   \sum_{j=1}^m d_j v_j( B_j \Lc W_T+ B_j \Lc U_T) -  \frac{1}{2} \sum_{j=1}^d d_j \|B_j \Lc U\|^2_{ \mathbb{H}(E^j, B_j \Lc \Pi \Lc B_j^*)} \right]\\
&\leq \sum_{j=1}^m d_j \log \EE e^{v_j(B_j \Lc  W_T)} .
\end{align*}
Since $f_i \leq e^{u_i}$ and $e^{v_j} = g_j + C \delta^c$, we conclude by arbitrary choice of $\epsilon, \delta$ that 
$$
\prod_{i=1}^k \left( \EE [f_i(W_T^{i})] \right)^{c_i} \leq \prod_{j=1}^m \left( \EE [g_j(B_j \Lc  W_T)]\right)^{d_j}. 
$$
In particular, writing out the expectations as integrals and canceling common factors using \eqref{eq:consistency},  the above may be rewritten as
\begin{align*}
&\prod_{i=1}^k \left(  \int_{E_i} f_i(x) e^{-\langle V_i x,x\rangle/2T} dx \right)^{c_i} \\
&\leq \frac{\prod_{i=1}^k \det(V^{-1}_i)^{1/2} }{  \prod_{j=1}^m \det(B_j \Lc \Pi \Lc B_j^*)^{1/2} } \prod_{j=1}^m \left(  \int_{E^j} g_j(x) e^{-\langle (B_j \Lc \Pi \Lc B_j^*)^{-1}x,x\rangle/2T} dx \right)^{d_j} .%
\end{align*}
Letting $T\to +\infty$, monotone convergence yields 
\begin{align*}
&\prod_{i=1}^k \left(  \int_{E_i} f_i(x)  dx \right)^{c_i} \leq \left( \frac{  \prod_{j=1}^m \det((B_j \Lc \Pi \Lc B_j^*)^{-1})^{d_j} }{\prod_{i=1}^k \det(V_i)^{c_i} }  \right)^{1/2}\prod_{j=1}^m \left(  \int_{E^j} g_j(x)   dx \right)^{d_j}. 
\end{align*}
The consequences of this are two-fold: defining $U_j = (B_j \Lc \Pi \Lc B_j^*)^{-1}$, $1\leq j\leq m$, we see that \eqref{propOptim} implies $(V_i)_{1\leq i\leq k},(U_j)_{1\leq j \leq m}$ satisfy \eqref{Qconstraint}.  Therefore, the ratio of determinants is at most $\exp({D_g(\mathbf{c},\mathbf{d},\mathbf{B})})$, proving  Theorem \ref{thm:FRBLunderGaussExt}.   In fact, the ratio of determinants must be precisely  equal to $\exp({D_g(\mathbf{c},\mathbf{d},\mathbf{B})})$, since we have freedom in choosing the functions $(f_i)_{1\leq i\leq k},(g_j)_{1\leq j \leq m}$ subject to \eqref{FRBLhypIntro}.  Thus, we have also shown that if \eqref{eq:consistency} holds and there exist $V_i \in S^+(E_i)$, $i=1, \dots, k$ and $\Pi \in \Pi(V^{-1}_1,\dots, V^{-1}_k)$ satisfying \eqref{propOptim}, then 
the datum $(\mathbf{c},\mathbf{d},\mathbf{B})$ is gaussian-extremizable with $(V_i)_{1\leq i\leq k}$ and $U_j = (B_j \Lc \Pi \Lc B_j^*)^{-1}$, $1\leq j\leq m$ being gaussian extremizers. This proves  the converse of  Proposition \ref{prop:OptimCond} (and therefore Theorem \ref{thm:FRBLgaussianExtiff} on the basis of  Proposition \ref{prop:OptimCond}). 
\end{proof}

\begin{remark}
Using the structure of gaussian extremizers in the gaussian-extremizable case, other proof techniques such as optimal transport or heat flow should also work to establish the above.  
\end{remark}

\subsection{Proof of Theorem \ref{thm:FRBLdualityConstants}}

For $V_i\in S^+(E_i), 1\leq i \leq  k$ and $U_j\in S^+(E^j), 1\leq j \leq m$, we first note that the Schur complement condition for positive-semidefiniteness implies
$$
\Lc^{1/2} B^* \Ld^{1/2}\left( \oplus_{j=1}^m U_j  \right)  \Ld^{1/2}B \Lc^{1/2} \leq \left( \oplus_{i=1}^k V_i  \right)
$$
if and only if 
$$
 \Ld^{1/2}B \Lc^{1/2}  \left( \oplus_{i=1}^k V^{-1}_i  \right) \Lc^{1/2} B^* \Ld^{1/2} \leq \left( \oplus_{j=1}^m U^{-1}_j  \right). 
$$
Noting that the first inequality is precisely \eqref{operatorInequality}, it follows immediately that 
$D_g(\mathbf{c},\mathbf{d},\mathbf{B}) = D_g(\mathbf{d},\mathbf{c},\mathbf{B}^*)$ (regardless of gaussian-extremizability).  Now, suppose the datum $(\mathbf{c},\mathbf{d},\mathbf{B})$ is gaussian-extremizable.  If $V_i\in S^+(E_i),  1\leq i \leq  k$ and $U_j\in S^+(E^j),  1\leq j \leq   m$ are gaussian extremizers for $(\mathbf{c},\mathbf{d},\mathbf{B})$, then it is immediate from the above observation that $U^{-1}_j\in S^+(E^j),  1\leq j \leq  m$ and $V^{-1}_i\in S^+(E_i),  1\leq i \leq  k$ and  are gaussian extremizers for the datum $(\mathbf{d},\mathbf{c},\mathbf{B}^*)$.

\subsection{Proof of Proposition \ref{prop:OptimCond}} \label{sec:ProofofOptimalityProp}
The goal of this section is to prove the optimality conditions asserted in Proposition \ref{prop:OptimCond}, which was the core assumption needed to prove Theorems \ref{thm:FRBLgaussianExtiff} and \ref{thm:FRBLunderGaussExt}.  The basic argument boils down to a strong min-max theorem.  This is given in the first subsection.  The second subsection leverages this min-max identity to complete the proof of Proposition \ref{prop:OptimCond}.  The arguments of this section follow those appearing in our previous work \cite{liu2018forward}, however it suffices to restrict attention to a particular duality enjoyed by positive-definite operators.

\subsubsection{A strong min-max theorem}

\begin{theorem}\label{thm:FRdualQuadraticForms} For any $K_i \in S^+(E_i)$, $i=1, \dots, k$, it holds that
\begin{align}
&\max_{K \in \Pi(K_1, \dots, K_k) }\sum_{j=1}^m d_j     \log \det \left( B_j \Lc K \Lc B_j^* \right) + \sum_{j=1}^m d_j  \dim(E^j) \notag \\
&=\inf_{(V_i,U_j)_{1\leq i\leq k, 1\leq j \leq m}}  \left( 
 \sum_{i=1}^k c_i  \langle V_i, K_i\rangle_{\HS} - \sum_{j=1}^m d_j \log \det U_j\right) , \label{FenchelMaxCoupling}
\end{align}
where the infimum is over $V_i\in S^+(E_i),1\leq i\leq k$ and $U_j\in S^+(E^j), 1\leq j\leq m$ satisfying  $\Lc B^* U_{\mathbf{d}} B \Lc  \leq V_{\mathbf{c}}.$
\end{theorem}

The critical ingredient in the proof is   the Fenchel-Rockafellar duality theorem \cite{rockafellar1966extension}, stated here as it appears in \cite[Theorem 1.9]{villani2003topics}:
 \begin{theorem}\label{thm:Fenchel-Rockafellar }
 Let $X$ be a normed vector space, $X^*$ its topological dual space, and $\Theta, \Xi$ two proper convex functions  Let $\Theta^*, \Xi^*$ be the Legendre-Fenchel transforms of $\Theta, \Xi$ respectively.  Assume that there is some $x_0\in X$ such that 
 $$
 \Theta(x_0) < +\infty, ~~\Xi(x_0)< +\infty,
 $$
 and $\Theta$ is continuous at $x_0$.  Then, 
 $$
 \inf_{x\in X}\Big[\Theta(x) + \Xi(x) ] = \max_{x^*\in X^*}\Big[-\Theta^*(-x^*) - \Xi^*(x^*) \Big].
 $$
 \end{theorem}
 \begin{remark}
 It is a part of both theorems above that the stated maximum is attained. 
 \end{remark}

 \begin{proof}[Proof of Theorem \ref{thm:FRdualQuadraticForms}]  
In our application of the Fenchel-Rockafellar theorem, we will take $X = S(E_0)$, regarded as a Hilbert space with respect to the usual Hilbert-Schmidt inner product.  In this case,  we also identify $X^* = S(E_0)$.   So, for $K_i \in S^+(E_i)$, $i=1, \dots, k$, given and $F \in S(E_0)$,   define the functionals
 $$
 \Xi(F) := \inf_{
 \substack{V_1\in S^+(E_1), \dots, V_k \in S^+(E_k) : \\  \langle V_{\mathbf{c}} x, x\rangle  \geq \langle Fx,x\rangle, ~\forall x\in E_0
 }}  \sum_{i=1}^k c_i \langle V_i, K_i\rangle_{\HS}.
 $$
 and
 $$
 \Theta(F) := \inf_{
 \substack{U_1\in S^+(E^1), \dots, U_m \in S^+(E^j): \\  \langle U_{\mathbf{d}}   B \Lc x  ,  B \Lc  x  \rangle  \leq \langle F x,x\rangle, ~\forall x\in E_0
 }}  \sum_{j=1}^m d_j \log \det U_j^{-1},
 $$
 with the convention that $\Theta(F) = +\infty$ if $F\notin S^+(E_0)$ (since the infimum will be over an empty set). 

 It is easy to see that both $\Theta :  S(E_0)\to \R$ and $\Xi :  S(E_0)\to \R$ are proper convex functions, with $\Theta(\operatorname{id}_{E_0})<+\infty$ and $\Xi(\operatorname{id}_{E_0}) <+\infty$.  It is straightforward to  check the continuity of $\Theta$ at $\operatorname{id}_{E_0}$, so the  hypotheses of Theorem \ref{thm:Fenchel-Rockafellar } are fulfilled.

  Let $M_0$ denote the infimum in the RHS of \eqref{FenchelMaxCoupling}, and observe that definitions easily imply 
   \begin{align*}
 &\inf_{F\in S(E_0) }\Big[\Theta(F) + \Xi(F) ] = M_0.
 \end{align*}
So, we turn our attention  toward computing 
$$\max_{H \in S( E_0)}\Big[-\Theta^*(-H) - \Xi^*(H) \Big].$$
  To this end, we claim that 
\begin{align}
 \Xi^*(H) \geq 
  \begin{cases}
   0 & \mbox{if}~H  \in \Pi(K_1, \dots, K_k) \\
   +\infty & \mbox{otherwise.}
   \end{cases} \label{XiConj}
\end{align}
Indeed, 
 \begin{align*}
   \Xi^*(H) = \sup_{F \in S(E_0)} \left( \langle H,F\rangle - \Xi(F)\right) &\geq   \sup_{F \in S^+(E_0) } \left( \langle H,- F\rangle - \Xi(- F)\right) \\
   &=   \sup_{F \in S^+(E_0) } - \langle H, F\rangle
   =  \begin{cases}
   0 & \mbox{if}~H\in S^+(E_0) \\
   +\infty & \mbox{otherwise.}
   \end{cases}
\end{align*}
Next, define $H_i := \pi_{E_i} H \pi^*_{E_i}$ and note that 
 \begin{align*}
   \Xi^*(H) = \sup_{F \in S(E_0)} \left( \langle H,F\rangle - \Xi(F)\right) 
   &\geq \sup_{F_i \in S(E_i) } \left( \langle H, \pi_{E_i}^* F_i \pi_{E_i} \rangle - \Xi(\pi_{E_i}^* F_i \pi_{E_i} ) \right) \\
   &= \sup_{F_i \in S(E_i) }  \langle H_i - K_i , F_i\rangle  = 
   \begin{cases}
   0 & \mbox{if}~H_i = K_i \\
   +\infty & \mbox{otherwise.}
   \end{cases}
\end{align*}
 This proves \eqref{XiConj}.   Hence, we may conclude
 \begin{align*}
 &\max_{H \in S( E_0)}\Big[-\Theta^*(-H) - \Xi^*(H) \Big] \\
 &\leq   \sup_{K \in \Pi(K_1, \dots, K_k)}\Big[-\Theta^*(-K)   \Big] \\
 &=    \sup_{K \in \Pi(K_1, \dots, K_k)} \inf_{F\in S^+(E_0)} \left[\langle K,F\rangle + \Theta(F) \right]\\
 &=  \sup_{K \in \Pi(K_1, \dots, K_k)}  \inf_{(U_j\in S^+(E^j))_{1\leq j\leq m}} 
\left(  \sum_{j=1}^m d_j \langle K,  \Lc B_j^* U_j B_j \Lc \rangle   -\sum_{j=1}^m d_j \log \det U_j\right) \\
  &=  \sup_{K \in \Pi(K_1, \dots, K_k)}   \sum_{j=1}^m d_j 
  \left( \inf_{(U_j\in S^+(E^j))_{1\leq j\leq m}} \left( 
 \langle B_j \Lc K \Lc B_j^*  ,  U_j   \rangle   -\sum_{j=1}^m d_j \log \det U_j\right) \right) \\
 &= \sup_{K \in \Pi(K_1, \dots, K_k)}   \sum_{j=1}^m d_j     \log \det \left( B_j \Lc K \Lc B_j^* \right) + \sum_{j=1}^m d_j  \dim(E^j) . 
   \end{align*}
It is straightforward to argue that that the supremum is attained using the facts that  $\Pi(K_1, \dots, K_k)$ is compact and that $B_j \Lc K \Lc B_j^*$ is uniformly bounded over all $K\in \Pi(K_1, \dots, K_k)$.   

 Therefore, invoking Theorem \ref{thm:Fenchel-Rockafellar }, we have shown
 \begin{align*}
M_0 
 &\leq \max_{K \in \Pi(K_1, \dots, K_k)}   \sum_{j=1}^m d_j     \log \det \left( B_j \Lc K \Lc B_j^* \right) + \sum_{j=1}^m d_j  \dim(E^j). 
 \end{align*}
   
The reverse direction is considerably simpler; consider any $K \in \Pi(K_1, \dots, K_k)$.  Then, 
\begin{align}
&\sum_{j=1}^m d_j     \log \det \left( B_j \Lc K \Lc B_j^* \right) + \sum_{j=1}^m d_j  \dim(E^j) \notag \\
&\leq \inf_{(U_j\in S^+(E^j))_{1\leq j\leq m} }\sum_{j=1}^m d_j \langle  B_j \Lc K  \Lc B_j^*, U_j \rangle - \sum_{j=1}^m d_j \log \det U_j. \label{almostDone}
\end{align}
Note that, for any $U_j\in S^+(E^j),1\leq j\leq m$, if $V_i\in S^+(E_i),1\leq i\leq k$   satisfy
\begin{align}
\langle U_{\mathbf{d}}   B \Lc x  ,  B \Lc  x  \rangle     \leq  \langle V_{\mathbf{c}} x, x\rangle   ~~~~\forall x\in  E_0, \label{quadGauss}
\end{align}
then it will hold that
\begin{align}
\sum_{j=1}^m d_j \langle  B_j \Lc K \Lc B_j^*, U_j \rangle_{\HS} \leq  \sum_{i=1}^k c_i  \langle V_i, K_i\rangle_{\HS}. \label{CovEst}
\end{align}
To see that this is indeed the case, let $x$ be a centered gaussian random vector in $E_0$ with covariance $K$, and take expectations of both sides in  \eqref{quadGauss}.   So, combining estimates yields
$$
\sum_{j=1}^m d_j     \log \det \left( B_j \Lc K \Lc B_j^* \right) + \sum_{j=1}^m d_j  \dim(E^j)\leq M_0,
$$
completing the proof of the theorem. 
\end{proof}

\subsubsection{Completion of proof of Proposition \ref{prop:OptimCond}}

The first step in completing the proof of Proposition \ref{prop:OptimCond} is to note an equivalent dual formulation of the optimization problem defining $D_g(\mathbf{c},\mathbf{d},\mathbf{B})$.  This formulation relies on the strong min-max identity of the previous subsection.  Through an analysis of the equality cases, we ultimately arrive at Proposition \ref{prop:OptimCond}.

\begin{theorem}\label{thm:equivFormsDg}
Fix $-\infty < C \leq + \infty$.  The following statements are equivalent:
\begin{enumerate}[1)]
\item 
For all $V_i\in S^+(E_i),1\leq i\leq k$ and $U_j\in S^+(E^j), 1\leq j\leq m$ satisfying 
\begin{align}
\Lc B^* U_{\mathbf{d}} B \Lc  \leq V_{\mathbf{c}} \label{positivityIneq}
\end{align}
it holds that
\begin{align}
\sum_{j=1}^m d_j \log \det U_j \leq C + \sum_{i=1}^k c_i \log \det V_i. \label{detObjective}
\end{align}

\item For all $K_i\in S^+(E_i),1\leq i\leq k$,
\begin{align}
 \sum_{i=1}^k c_i \log \det K_i \leq C +  \max_{K\in \Pi(K_1, \dots, K_k)} \sum_{j=1}^m d_j \log \det (B_j \Lc K \Lc B_j^*). \label{couplingObjective}
 \end{align}
 \end{enumerate}
 \end{theorem}
 \begin{remark}\label{rmk:Ceq2D}
By definition of $D_g(\mathbf{c},\mathbf{d},\mathbf{B})$ and the asserted equivalence, the best constant $C$ in each of the above inequalities is equal to $2 D_g(\mathbf{c},\mathbf{d},\mathbf{B})$.
 \end{remark}
\begin{remark}\label{rmk:differenceFromClassical}
To appreciate the difference between the forward-reverse Brascamp-Lieb inequality and the classical forward and reverse inequalities, we invite the reader to prove Theorem \ref{thm:equivFormsDg} in the special case of $k=1$ (which is symmetric with $m=1$).  This is a simple exercise, requiring only a few lines of elementary linear algebra.  Since the maximum in \eqref{couplingObjective} becomes trivial, the major difficulty of the characterization (i.e., the strong min-max theorem of the previous section) is avoided.  
\end{remark}
 \begin{proof}
 For purposes of the proof, we assume each   $(B_j)_{1\leq j \leq m}$ is surjective, ensuring invertibility of $(B_j \Lc K \Lc B_j^*)$ for $K\in S^+(E_0)$. If any of the $B_j$ fail to be surjective, it is easy to see that the best constant $C$ in both 1) and 2) will be equal to $+\infty$, thereby handling this exceptional case.  Moreover, we may assume that  \eqref{eq:consistency} holds.  Again, if this is not the case, then by rescaling the various operators by a common factor, we see that the best constant $C$ in both 1) and 2) will be equal to $+\infty$.

\emph{Proof of 1)$\Rightarrow$2)}. Fix any $\epsilon>0$.  By Theorem  \ref{thm:FRdualQuadraticForms}, there are $V_i\in S^+(E_i),1\leq i\leq k$ and $U_j\in S^+(E^j), 1\leq j\leq m$ satisfying \eqref{positivityIneq} such that 
\begin{align*}
&\max_{K\in \Pi(K_1, \dots, K_k)} \sum_{j=1}^m d_j \log \det (B_j \Lc K \Lc B_j^*) + \sum_{j=1}^m d_j \dim(E^j)\\
&\geq  \sum_{i=1}^k c_i \langle V_i ,K_i\rangle_{\HS}-\sum_{j=1}^m d_j \log \det U_j -   \epsilon.
\end{align*}
Hence, we have
\begin{align*}
&\sum_{i=1}^k c_i \log \det K_i - \max_{K\in \Pi(K_1, \dots, K_k)} \sum_{j=1}^m d_j \log \det (B_j \Lc K \Lc B_j^*)\\
&\leq \sum_{i=1}^k c_i \log \det K_i + \sum_{j=1}^m d_j \dim(E^j) +\sum_{j=1}^m d_j \log \det U_j - \sum_{i=1}^k c_i \langle V_i ,K_i\rangle_{\HS} + \epsilon\\
&\leq \sum_{i=1}^k c_i \log \det K_i + \sum_{j=1}^m d_j \dim(E^j)  + C + \sum_{i=1}^k c_i \log \det V_i - \sum_{i=1}^k c_i \langle V_i ,K_i\rangle_{\HS} + \epsilon\\
&\leq C + \epsilon,
\end{align*}
where the penultimate inequality is \eqref{detObjective}, and the final inequality is due to the elementary inequality $\log\det M \leq \Tr(M)-\dim(E_i)$ for $M\in S^+(E_i)$ and the scaling condition \eqref{eq:consistency}. Since $\epsilon$ was arbitrary, it follows that 1)$\Rightarrow$ 2). 

\emph{Proof of 2)$\Rightarrow$1)}.
Fix any $V_i\in S^+(E_i), i=1, \dots, k$.  By the method of Lagrange multipliers and the weak max-min inequality, we have
\begin{align*}
&\sup_{\substack{(U_j\in S^+(E^j))_{1\leq j\leq m} :\\
\Lc B^* U_{\mathbf{d}} B \Lc  \leq V_{\mathbf{c}}  } }  ~~\sum_{j=1}^m d_j \log \det U_j \notag \\
&\leq \inf_{M \in S^+(E_0) }~~ \sup_{( U_j\in S^+(E^j))_{1\leq j\leq m}}  ~  \left[ \sum_{j=1}^m d_j \log \det U_j  + \langle M ,  V_{\mathbf{c}} \rangle_{\HS}   - \langle M ,  \Lc B^* U_{\mathbf{d}} B \Lc \rangle_{\HS}   \right]\\
&= \inf_{M \in S^+(E_0) }~~ \sup_{( U_j\in S^+(E^j))_{1\leq j\leq m}}  ~  \left[ \sum_{j=1}^m d_j \log \det U_j  + \langle M ,   V_{\mathbf{c}} \rangle_{\HS}  - \sum_{j=1}^m d_j \langle B_j \Lc M \Lc B_j^* ,   U_j \rangle_{\HS}   \right].
\end{align*}
Now, we note that the gradient of the objective above with respect to $U_j$ is given by $d_j U_j^{-1} - B_j \Lc M \Lc B_j^*$.  So, taking $U_j = (B_j \Lc M \Lc B_j^*)^{-1}$ achieves the maximum in the inner optimization problem. Making this substitution, we may continue as
\begin{align}
&\sup_{\substack{(U_j\in S^+(E^j))_{1\leq j\leq m} :\\
\Lc B^* U_{\mathbf{d}} B \Lc  \leq V_{\mathbf{c}}  } }  ~~\sum_{j=1}^m d_j \log \det U_j  \notag \\
&= \inf_{M \in S^+(E_0) }~  \sum_{j=1}^m d_j \log \det \left( B_j \Lc M \Lc B_j^*\right)^{-1}  - \sum_{j=1}^m d_j \dim(E^j)    + \langle M ,  V_{\mathbf{c}} \rangle_{\HS} \label{critObjective}\\
&\leq - \max_{M \in \Pi(V^{-1}_1,\dots, V^{-1}_k)}  \sum_{j=1}^m d_j \log \det \left( B_j \Lc M \Lc B_j^*\right) \label{criticalInequality}\\
&\leq C - \sum_{i=1}^k c_i \log \det V^{-1}_i = C + \sum_{i=1}^k c_i \log \det V_i  \notag
\end{align}
where the first inequality follows since we are optimizing over a smaller set, and for $M \in \Pi(V^{-1}_1,\dots, V^{-1}_k)$ it holds that 
$$ \langle M ,  V_{\mathbf{c}} \rangle_{\HS} = \sum_{i=1}^k c_i   \langle V_i^{-1},V_i\rangle_{\HS} = \sum_{i=1}^k c_i \dim(E_i) =  \sum_{j=1}^m d_j \dim(E^j).$$
 The second inequality is an application of \eqref{couplingObjective}. 
\end{proof}

Finally, we are in a position to complete the proof of Proposition \ref{prop:OptimCond}.
\begin{proof}[Proof  of Proposition \ref{prop:OptimCond}]
If the datum $(\mathbf{c},\mathbf{d},\mathbf{B})$ is gaussian-extremizable, then \eqref{eq:consistency} must hold, as remarked at the start of the proof of Theorem \ref{thm:equivFormsDg}.  Furthermore, for  extremizers $V_i\in S^+(E_i),  1\leq i \leq  k$, we get equality throughout the second part of the proof of Theorem \ref{thm:equivFormsDg} for the optimal constant $C = 2D_g(\mathbf{c},\mathbf{d},\mathbf{B})$; in particular, equality will be attained in \eqref{criticalInequality}.  Therefore, we remark that any
$$
\Pi \in \arg  \max_{M \in \Pi(V^{-1}_1,\dots, V^{-1}_k)}  \sum_{j=1}^m d_j \log \det \left( B_j \Lc M \Lc B_j^*\right)
$$
is  positive-semidefinite and achieves the minimum in \eqref{critObjective}, after replacing the infimum over positive-definite operators with a minimum over  positive-semidefinite operators.  Letting $\Pi$ be as defined above, consider any positive-semidefinite $A\in S(E_0)$. For  $\varepsilon>0$, Taylor expansion and assumed optimality of $\Pi$ gives 
\begin{align*}
&\sum_{j=1}^m d_j \log \det \left( B_j \Lc (\Pi + \varepsilon A) \Lc B_j^*\right)^{-1}  - \sum_{j=1}^m d_j \dim(E^j)    + \langle (\Pi + \varepsilon A) ,  V_{\mathbf{c}} \rangle_{\HS}\\
&=\sum_{j=1}^m d_j \log \det \left( B_j \Lc \Pi \Lc B_j^*\right)^{-1}  - \sum_{j=1}^m d_j \dim(E^j)    + \langle \Pi,  V_{\mathbf{c}} \rangle_{\HS}\\
&+\varepsilon  \left( -\sum_{j=1}^m d_j \left\langle A , \Lc B_j^*\left( B_j \Lc \Pi \Lc B_j^*\right)^{-1} B_j \Lc     \right\rangle_{\HS} + \langle A,  V_{\mathbf{c}} \rangle_{\HS}\right)  + o(\varepsilon)\\
&\geq \sum_{j=1}^m d_j \log \det \left( B_j \Lc \Pi   \Lc B_j^*\right)^{-1}  - \sum_{j=1}^m d_j \dim(E^j)    + \langle  \Pi   ,  V_{\mathbf{c}} \rangle_{\HS}.
\end{align*}
By sending $\varepsilon \downarrow 0$ and letting $A\geq 0$ vary, we find that it must hold that 
\begin{align}
\sum_{j=1}^m d_j \Lc B_j^*\left( B_j \Lc \Pi \Lc B_j^*\right)^{-1} B_j \Lc \leq V_{\mathbf{c}} \label{claimedOptimality}
\end{align}
as claimed. 

\begin{remark}\label{rmk:restrictionOfGaussExtremizableCondtions}
If we  consider $A$ such that $\ker(A)\subseteq \ker(\Pi)$, then the same inequalities above hold for all $\varepsilon \in \R$ sufficiently small, since this ensures $\Pi +\varepsilon A$ will remain positive-semidefinite.  As a result, we have equality in \eqref{claimedOptimality} on the restriction of both sides to $\ker(\Pi)^{\perp}$. This provides an alternate way to establish the conclusion of Remark \ref{rmk:restrictSubspace}.
\end{remark}
\end{proof}

Although it will not be needed for our purposes,  there are several equivalent characterizations of gaussian extremizers which can be stated in analogy to  \cite[Proposition 3.6]{bennett2008brascamp} for the direct Brascamp-Lieb inequality.  The statement is provided below to give a  comprehensive description of the structure of gaussian extremizers. It provides a convenient means of certifying the optimality of gaussian extremizers, which is required to compute $D_g(\mathbf{c},\mathbf{d},\mathbf{B})$.

\begin{proposition}\label{prop:EquivOptimalConds}
Fix a datum $(\mathbf{c},\mathbf{d},\mathbf{B})$, and $(U_j \in S^+(E^j) )_{1\leq j \leq m}, (V_i \in S^+(E_i))_{1\leq i\leq k}$.   The following statements are equivalent:
\begin{enumerate}[(i)]
\item\label{l1}  $(U_j)_{1\leq j \leq m}, (V_i)_{1\leq i\leq k}$ is a global maximum in \eqref{defDg} subject to \eqref{Qconstraint};
\item\label{l2}    $(U_j)_{1\leq j \leq m}, (V_i)_{1\leq i\leq k}$ is a local maximum in \eqref{defDg} subject to \eqref{Qconstraint};
\item\label{l3} There exists $\Pi \in \Pi(V_1^{-1}, \dots, V_k^{-1})$ such that
\begin{align}
U_j^{-1} = B_j \Lc \Pi \Lc B_j^*, ~~~1 \leq j\leq m. \label{UjCondOpt}
\end{align}
Moreover, for 
\begin{align}
\Theta := V_{\mathbf{c}} - \sum_{j=1}^m d_j \Lc B_j^* U_jB_j\Lc ,  \label{ThetaDefOpt}
\end{align}
we have $\Theta\geq 0$ and $ \Theta x =0$ for any $x\in \Pi E_0$. 
\item\label{l4} The dimension condition \eqref{eq:consistency} is satisfied, and there exists  $\Pi \in \Pi(V_1^{-1}, \dots, V_k^{-1})$ such that \eqref{UjCondOpt} holds. Moreover, $\Theta\geq 0$, with $\Theta$ defined as in \eqref{ThetaDefOpt}. 
\item\label{l5} The dimension condition \eqref{eq:consistency} is satisfied, and there exists  $\Pi \in \Pi(V_1^{-1}, \dots, V_k^{-1})$ such that 
\begin{align}
U_j^{-1} \geq B_j \Lc \Pi \Lc B_j^* , ~~~1 \leq j\leq m. \label{e3}
\end{align}
Moreover, $\Theta\geq 0$ and $\Pi^{1/2} \Theta \Pi^{1/2}=0$, with $\Theta$ defined as in \eqref{ThetaDefOpt}. 
\end{enumerate}

\end{proposition}
\begin{proof} 

\eqref{l1}$\Longrightarrow$\eqref{l2} is trivial.

\eqref{l3}$\Longrightarrow$\eqref{l4}: By the assumption of \eqref{l3} and the fact that $\Pi^{1/2}E_0 = \Pi E_0$, we have $\Tr(\Pi^{1/2}\Theta\Pi^{1/2})=0$.  Using \eqref{UjCondOpt} and \eqref{ThetaDefOpt} and the cyclic property of the trace, we find
\begin{align}
\sum_{i=1}^k c_i\dim(E_i)=\sum_{j=1}^m d_j\dim(E^j).\label{dimCondEqCon}
\end{align}

\eqref{l4}$\Longrightarrow$\eqref{l3}:  It suffices to show that $\Pi^{1/2}\Theta\Pi^{1/2}=0$.
 As argued above, the dimension condition is equivalent to $\Tr(\Pi^{1/2}\Theta\Pi^{1/2})=0$.
Moreover, the assumption of \eqref{l4} requires that $\Pi^{1/2}\Theta\Pi^{1/2}\geq 0$.
Since a nonnegative matrix has trace 0 only if it is the zero matrix, we have $\Pi^{1/2}\Theta\Pi^{1/2}=0$ as desired.

\eqref{l3}$\Longrightarrow$\eqref{l1}: This is the conclusion of Section \ref{sec:Lehec-Style}.

\eqref{l3}$\Longrightarrow$\eqref{l5}: The proof of the dimension condition is the same as the \eqref{l3}$\Longrightarrow$\eqref{l4} part.

\eqref{l5}$\Longrightarrow$\eqref{l3}: By the assumption of \eqref{l5} we have 
\begin{align}
0=\Tr(\Pi^{1/2}\Theta\Pi^{1/2})
&=\sum_{i=1}^k c_i\dim(E_i)-
\sum_{j=1}^md_j\Tr(U_j(B_j\Lambda_{\bf c}\Pi\Lambda_{\bf c}B_j^*))\notag
\\
&\ge \sum_{i=1}^k c_i\dim(E_i)-
\sum_{j=1}^md_j\Tr(U_j U_j^{-1})
\label{e10}
\\
&=0 \notag
\end{align}
where \eqref{e10} follows from \eqref{e3}.  If \eqref{e3} is not equality for some $j$, then \eqref{e10} cannot achieve equality. Therefore \eqref{e3} must achieve equality for all $j$.

\eqref{l2}$\Longrightarrow$\eqref{l3}:  By definition, we have  
\begin{align}
D_g(\mathbf{c},\mathbf{d},\mathbf{B}) := \frac{1}{2}\sup \left( \sum_{j=1}^m d_j \log \det U_j - \sum_{i=1}^k c_i \log \det V_i\right),  \label{DgOpt}
\end{align}
where the supremum is over $V_i \in S^+(E_i)$ and $ U_j \in S^+(E^j)$ satisfying the constraint  
\begin{align}
\Theta:= V_{\bf c}
- \Lambda_{\bf c}B^*
U_{\bf d}
B\Lambda_{\bf c}   \geq 0 .
\label{e4}
\end{align}
We need to show that if given $(V_i)_{1\leq i\leq k }$ and $(U_j)_{1\leq j\leq m}$ is a local maximum of \eqref{DgOpt},
then there exists $\Pi$ with the claimed properties. To this end, let  $S_0^+(E_0)$ denote the set of positive semidefinite operators on $E_0$,  and define the  convex cone:
\begin{align}
\mathcal{C}:=
\{(c_1\Pi_1,\dots,c_k\Pi_k,-d_1B_1\Lambda_{\bf c}\Pi\Lambda_{\bf c} B_1^*,\dots,-d_mB_m \Lambda_{\bf c}\Pi\Lambda_{\bf c} B_m^* )\,\colon \Pi\in S_0^+(E_0) \},
\end{align}
where $\Pi_i: E_i \to E_i$ is defined as $\Pi_i = \pi_{E_i}\Pi \pi^*_{E_i}$.
Recall that the dual cone $\mathcal{C}^*$ of $\mathcal{C}$ is defined as the set of all vectors whose inner product with any element in $\mathcal{C}$ is nonnegative.
Here, we have
\begin{align}
\mathcal{C}^*
:=\{(V_1,\dots,V_k,
U_1,\dots,U_m)\colon\,
V_i\in S^+_0(E_i) ,
U_j\in S^+_0(E^j),
V_{\bf c}-\Lambda_{\bf c}B^*U_{\bf d}B\Lambda_{\bf c}
\geq 0
\}.
\end{align}
Note that this is indeed the dual cone since the constraint $V_{\bf c}-\Lambda_{\bf c}B^*U_{\bf d}B\Lambda_{\bf c}
\geq 0$ can be rewritten as 
\begin{align}
\left\langle\Pi, V_{\bf c}-\Lambda_{\bf c}B^*U_{\bf d}B\Lambda_{\bf c}\right\rangle_{\HS }\geq 0, \hspace{1cm} \mbox{for all~}\Pi\in S^+_0(E_0), \label{ThGEQ}
\end{align}
or equivalently,
$\sum_{i=1}^k\left\langle c_i\Pi_i,V_i\right\rangle_{\HS}
+\sum_{j=1}^m\left\langle -d_j
B_j\Lambda_{\bf c}\Pi\Lambda_{\bf c}B_j^*,U_j\right\rangle_{\HS}\geq 0$, 
for all $\Pi\in S^+_0(E_0)$.
Note that the constraint defining $\mathcal{C}^*$ is the same as the constraint \eqref{e4} for the optimization \eqref{DgOpt}.
Now, from \eqref{DgOpt}, we see that the gradient of the objective function with respect to $(V_1,\dots,V_k,U_1,\dots,U_m)$ is equal to $(-\frac{c_1}{2}V_1^{-1},\dots,-\frac{c_k}{2}V_k^{-1},
\frac{d_1}{2}U_1^{-1},\dots,\frac{d_m}{2}U_m^{-1})$.
The assumed local optimality implies that the inner product of the gradient at $(V_i)_{1\leq i\leq k }$, $(U_j)_{1\leq j\leq m}$ with any element in $\mathcal{C}^*$ is non-positive,
hence the negative gradient belongs to the dual of $\mathcal{C}^*$.
Recall that the double-dual of a closed convex cone is itself,
hence the negative gradient belongs to $\mathcal{C}$,
and therefore (upon absorbing a factor $2$ in $\Pi$) we find $\Pi \in S_0^+(E_0)$ such that
\begin{align}
\Pi_i&= V_i^{-1},\quad i=1,\dots,k,
\label{e_8}
\\
B_j\Lambda_{\bf c}\Pi\Lambda_{\bf c}B_j^*
&=U_j^{-1},\quad j=1,\dots,m.
\label{e_9}
\end{align}
Note that \eqref{e_8} together with $\Pi \in S_0^+(E_0)$ is equivalent to $\Pi \in \Pi(V_1^{-1}, \dots, V_k^{-1})$, so it only remains to   show that  $\Theta x=0$ for $x\in \Pi E_0$, or equivalently $\Pi^{1/2} \Theta \Pi^{1/2}=0$.  The assumed local optimality implies that the dimension condition \eqref{dimCondEqCon} must hold, else scaling  $(V_i)_{1\leq i\leq k }$, $(U_j)_{1\leq j\leq m}$ by an appropriate common factor will increase the value of the functional being optimized. As noted in the proof of \eqref{l4}$\Longrightarrow$\eqref{l3}, the dimension condition together with \eqref{e4} (i.e., nonnegativity of $\Theta$)  implies $\Pi^{1/2} \Theta \Pi^{1/2}=0$, as desired.
\end{proof}

To conclude this section, we record the following observation which will be needed later.
\begin{lemma}\label{lem:semicontinuityOfgaussian}  Assume the $(B_j)_{1\leq j \leq m}$ are surjective.  
The map 
\begin{align}
 (K_1, \dots, K_k) \longmapsto \left( \sum_{i=1}^k c_i \log \det K_i -  \max_{K\in \Pi(K_1, \dots, K_k)} \sum_{j=1}^m d_j \log \det (B_j \Lc K \Lc B_j^*)\right) \notag
\end{align}
is upper-semicontinuous on $\prod_{i=1}^k S^+(E_i)$ with respect to the  norm topology.
\end{lemma}
\begin{proof}
The map $(K_1, \dots, K_k) \longmapsto  \sum_{i=1}^k c_i \log \det K_i$
is continuous on $\prod_{i=1}^k S^+(E_i)$, and we may write 
$$
\max_{K\in \Pi(K_1, \dots, K_k)} \sum_{j=1}^m d_j \log \det (B_j \Lc K \Lc B_j^*) = \sup_{K\in \Pi(K_1, \dots, K_k)\cap S^+(E_0) }  \sum_{j=1}^m d_j \log \det (B_j \Lc K \Lc B_j^*), 
$$
which is lower-semicontinuous in $(K_1, \dots, K_k)$ on $\prod_{i=1}^k S^+(E_i)$, since it is the pointwise supremum of continuous functions. 
\end{proof}

\section{Decomposability and Conditions for Finiteness}\label{sec:Decomposability}
The aim of this section is to complete the proof of Theorem \ref{thm:FRBL} by successively decomposing data which are not gaussian-extremizable into ones that are either gaussian-extremizable or degenerate.  A consequence of the arguments will be the necessary and sufficient conditions for finiteness given by  Theorem \ref{thm:necSuffCondDgFinite}.  The development closely parallels the treatment of the forward Brascamp-Lieb inequality in \cite{bennett2008brascamp}, however the modifications are significant enough that   it  warrants explicitly giving the details.  Our starting point will be to state a characterization of $D_g(\mathbf{c},\mathbf{d},\mathbf{B})$  in terms of Shannon entropies, denoted by $h$.  This characterization will be exploited to facilitate the various computations later on.  Basic properties of the Shannon entropy (subadditivity, scaling, etc.) will be taken for granted here; the unfamiliar reader can find the needed properties collected in Appendix \ref{app:Entropy}, or any standard text on information theory. 

\subsection{Entropic characterization of $D_g(\mathbf{c},\mathbf{d},\mathbf{B})$}\label{subsec:EntropyDg}
We first introduce some notation that will be needed below and again in Section \ref{subsec:AJN}.  For a collection of random vectors $(X_i)_{1\leq i\leq k}$ taking values in $(E_i)_{1\leq i\leq k}$, respectively, we denote the marginal law of each $X_i$ as $P_{X_i}$, and denote their joint law by $P_{\mathbf{X}}$.  The set of couplings of $(X_i)_{1\leq i\leq k}$ (i.e., joint laws of $(X_i)_{1\leq i\leq k}$ with $X_i$-marginal equal to $P_{X_i}$ for each $1\leq i\leq k$) is denoted by $\Pi(P_{X_1}, \dots, P_{X_k})$.  Since elements of $S^+(E_i)$ are in one-to-one correspondence with centered gaussian probability measures on $E_i$ (see Remark \ref{rmk:PiasCouplings}), this notation is consistent with the earlier definition of  $\Pi(K_1, \dots, K_k)$ for $K_i \in S^+(E_i)$, $1\leq i\leq k$.

If $Z$ is a gaussian random vector in $E$ with covariance $\Sigma \in S^{+}(E)$, we have the identity
\begin{align}
h(Z) = \frac{1}{2}\log\left( (2\pi e)^{\dim(E)} \det(\Sigma) \right).\label{gaussEntropyExpression}
\end{align}
  So, a reinterpretation of Theorem \ref{thm:equivFormsDg} and Remark \ref{rmk:Ceq2D} is the following entropic characterization of $D_g(\mathbf{c},\mathbf{d},\mathbf{B})$. A similar  characterization also holds for $D(\mathbf{c},\mathbf{d},\mathbf{B})$; see Theorem \ref{thm:entDualGeneral} and accompanying remarks in Section \ref{subsec:AJN}.
\begin{proposition}\label{prop:entropyCharacterizationDg}
If $(Z_i)_{1\leq i \leq k}$ are gaussian random vectors in $(E_i)_{1\leq i\leq k}$, respectively, then 
\begin{align}
\sum_{i=1}^kc_i  h\left( c_i^{-1} Z_i \right)  - \max_{P_{\mathbf{Z}}\in \Pi(P_{Z_1}, \dots, P_{Z_k})} \sum_{j=1}^m d_j h\left(\sum_{i=1}^k B_{ij} Z_i \right) \leq D_g(\mathbf{c},\mathbf{d},\mathbf{B}), \label{gaussianEntropyDg}
\end{align}
where the maximum is over all couplings of the $(Z_i)_{1\leq i \leq k}$.  Moreover, the constant $D_g(\mathbf{c},\mathbf{d},\mathbf{B})$ is best possible. 
\end{proposition}

\begin{remark}
Although not explicitly stated, it suffices to consider jointly gaussian couplings in \eqref{gaussianEntropyDg}, giving equivalence to \eqref{couplingObjective}.  This is a consequence of the fact that gaussians maximize entropy for a given covariance.  
\end{remark}
\begin{remark}\label{rmk:FiniteEntropies}
Since any choice of gaussian $(Z_i)_{1\leq i \leq k}$ in \eqref{gaussianEntropyDg} have finite second moments by definition, Proposition \ref{app:ExistEntropy} in Appendix \ref{app:Entropy} ensures that the entropies $h\left(\sum_{i=1}^k B_{ij} Z_i \right)$ exist and satisfy
$$
 h\left(\sum_{i=1}^k B_{ij} Z_i \right) <+\infty~~~\mbox{for each }1\leq j\leq m.
$$
Moreover, if $D_g(\mathbf{c},\mathbf{d},\mathbf{B})<+\infty$, then 
there must exist a coupling of the $(Z_i)_{1\leq i \leq k}$ (in particular, the one achieving the maximum in \eqref{gaussianEntropyDg})  such that 
\begin{align}
 h\left(\sum_{i=1}^k B_{ij} Z_i \right) >-\infty~~~\mbox{for each }1\leq j\leq m. \label{couplingFiniteEntropy}
\end{align}
In other words, under this optimal coupling, the entropy $h\left(\sum_{i=1}^k B_{ij} Z_i \right)$ exists and is finite for each $1\leq j\leq m$.
\end{remark}

\subsection{Decomposability  of $D(\mathbf{c},\mathbf{d},\mathbf{B})$ and $D_g(\mathbf{c},\mathbf{d},\mathbf{B})$ for non-simple data}

Let $T$ be a subspace of $E_0$ having product form, and define $B_{j,T}$ to be the restriction of $B_{j}$ to $T$.  Note that if $T_i:=\pi_{E_i} T$, and $B_{ij,T_i}$ is the restriction of $B_{ij}$ to $T_i$, then the product-form assumption implies
$$
B_{j,T} x = \sum_{i=1}^k {B_{ij,T_i}}\pi_{E_i}(x),~~~x\in T. 
$$
Now, let $B_{ij,T_i^{\perp}}$ denote the restriction of $(\pi_{(B_j T)^{\perp}} {B_{ij}})$ to $T_i^{\perp}$, the orthogonal complement of $T_i$ in $E_i$, and define 
the collections of linear maps
\begin{align*}
\mathbf{B}_T &:= \big\{B_{ij,T_i} : T_i \to (B_j T) \big\}_{1\leq i\leq k, 1\leq j\leq m}, \mbox{~~and} \\
\mathbf{B}_{E_0/  T} &:= \big\{   B_{ij,T_i^{\perp}}   : T^{\perp}_i   \to (B_j T)^{\perp}  \big\}_{1\leq i\leq k, 1\leq j\leq m},
\end{align*}
where $(B_j T)^{\perp} $ denotes the orthogonal complement of $B_j T$ in $E^j$, $1\leq j \leq m$. 

\begin{lemma}\label{lem:Dsubadditive}
For any product-form subspace $T\subseteq E_0$, it holds that 
$$
D(\mathbf{c},\mathbf{d},\mathbf{B}) \leq D(\mathbf{c},\mathbf{d},\mathbf{B}_T) +D(\mathbf{c},\mathbf{d},\mathbf{B}_{E_0/  T}) 
$$
and
$$
D_g(\mathbf{c},\mathbf{d},\mathbf{B}) \leq D_g(\mathbf{c},\mathbf{d},\mathbf{B}_T) +D_g(\mathbf{c},\mathbf{d},\mathbf{B}_{E_0/  T}) .
$$
\end{lemma}
\begin{remark}
It may happen in the decomposition of $(\mathbf{c},\mathbf{d},\mathbf{B})$ into $(\mathbf{c},\mathbf{d},\mathbf{B}_T)$ and $(\mathbf{c},\mathbf{d},\mathbf{B}_{E_0/T})$ that we encounter subspaces of dimension zero.  Just as argued in \cite{bennett2008brascamp}, these subspaces can be safely disregarded in the following and subsequent computations. In particular, the entropy of a random variable on a subspace of dimension zero (a degenerate situation) is defined to be equal to zero in subsequent computations. 
\end{remark}
\begin{proof} We assume both $D_g(\mathbf{c},\mathbf{d},\mathbf{B}_T)$ and $D_g(\mathbf{c},\mathbf{d},\mathbf{B}_{E_0/  T})$ are finite, else the claim is trivial.  

Fix $M>0$.  Let $f_i : E_i \to \R^+$, $1\leq i\leq k$ and $g_j : E^j \to \R^+$, $1\leq j\leq m$ be non-negative measurable  functions, bounded from above by $M$ and satisfying
\begin{align}
\prod_{i=1}^k f_i^{c_i}(z_i) \leq \prod_{j=1}^m g_j^{d_j}\left( \sum_{i=1}^k c_i B_{ij} z_i \right)\hspace{1cm}\forall z_i\in E_i,~ 1\leq i\leq k.\label{FRBLhypZ}
\end{align}
Define $T_i :=\pi_{E_i}T$.  For $z_i \in E_i$, define $x_i:= \pi_{T_i}(z_i)$ and $y_i:= \pi_{T^{\perp}_i}(z_i)$.  In this notation, the hypothesis \eqref{FRBLhypZ} implies 
\begin{align}
\prod_{i=1}^k f_i^{c_i}(x_i + y_i) \leq \prod_{j=1}^m g_j^{d_j}\left( \sum_{i=1}^k c_i \left(B_{ij,T_i} x_i   +  \pi_{B_{j} T}  B_{ij} y_i \right)+ \sum_{i=1}^k  c_i B_{ij,T_i^{\perp}} y_i  \right) .\label{FRBLhypX}
\end{align}
By the Fubini-Tonelli theorem, the map $x_i \in T_i \longmapsto f_i(x_i + y_i)$ is measurable for almost every $y_i \in T_i^{\perp}$; define $N_i\subset T_i^{\perp}$ to be the (null) set of $y_i \in T_i^{\perp}$ for which $x_i \in T_i \longmapsto f_i(x_i + y_i)$ is not measurable.  By defining 
$$
\tilde f_i(x_i + y_i) = f(x_i + y_i) 1_{T^{\perp}_i \setminus N_i}(y_i), \hspace{1cm}x_i \in T_i, y_i \in T_i^{\perp},
$$
we have that $x_i \in T_i \longmapsto \tilde f_i(x_i + y_i)$ is measurable for all $y_i \in T_i^{\perp}$.  Moreover, since $f$ was only modified on a null set, $\int_{E_i} \tilde f_i = \int_{E_i}   f_i$.   Similarly, the map $u_j \in B_j T \longmapsto g_j(u_j + v_j)$ is measurable for almost every $v_j \in (B_jT)^{\perp}$; define $N^j \subset (B_jT)^{\perp}$ to be the (null) set of $v_j \in (B_jT)^{\perp}$ for which $u_j \in B_j T \longmapsto g_j(u_j + v_j)$ is not measurable.  Almost as before, we define 
$$
\tilde g_j(u_j + v_j) =  g_j(u_j + v_j)1_{(B_j T)^{\perp} \setminus N^j}(v_j) + M 1_{ N^j}(v_j), 
\hspace{1cm}u_j \in B_j T, v_j \in (B_j T)^{\perp},
$$
and are guaranteed that $u_j \in B_j T \longmapsto \tilde g_j(u_j + v_j)$ is measurable for all $v_j \in (B_jT)^{\perp}$, and $\int_{E^j} \tilde g_j = \int_{E^j}   g_j$.  

Since $\tilde f_i \leq f_i$ and $g_j \leq \tilde g_j$ by construction, we have that \eqref{FRBLhypX} holds with $f_i$ (resp. $g_j$) replaced by $\tilde f_i$ (resp. $\tilde g_j$).  So, by definition of $D_g(\mathbf{c},\mathbf{d},\mathbf{B}_T)$ and translation invariance of the Lebesgue integral, we consider the hypothesis \eqref{FRBLhypX} for fixed $(y_i)_{1\leq i \leq k}$ to conclude 
$$
\prod_{i=1}^k \left( \int_{T_i} \tilde f_i(x + y_i) dx  \right)^{c_i} \leq   e^{D_g(\mathbf{c},\mathbf{d},\mathbf{B}_T)} 
\prod_{j=1}^m \left( \int_{B_jT} \tilde g_j\left(u + \sum_{i=1}^k  B_{ij,T_i^{\perp}} y_i  \right) du  \right)^{d_j} 
\forall y_i\in T^{\perp}_i,~ 1\leq i\leq k.\label{FRBLhypY}
$$
By Fubini-Tonelli, the map  $y_i \in T_i^{\perp} \longmapsto \int_{T_i} \tilde f_i(u + y_i) du$ is measurable.  Similarly, $v_j \in (B_jT)^{\perp} \longmapsto \int_{B_j T} \tilde g_j(u + v_j) du$ is measurable.   Therefore, by definition of $D_g(\mathbf{c},\mathbf{d},\mathbf{B}_{E_0/  T})$, we find\begin{align*}
\prod_{i=1}^k \left( \int_{T^{\perp}_i}  \int_{T_i} \tilde f_i(x + y) dx   dy \right)^{c_i} \leq   e^{D_g(\mathbf{c},\mathbf{d},\mathbf{B}_T) +D_g(\mathbf{c},\mathbf{d},\mathbf{B}_{E_0/  T}) } 
\prod_{j=1}^m \left( \int_{(B_jT)^{\perp}}   \int_{B_jT} \tilde g_j\left(u + v\right) du  dv \right)^{d_j}.
\end{align*}
By an application of Tonelli's theorem combined with the previously noted identities $\int_{E_i} \tilde f_i = \int_{E_i}   f_i$ and $\int_{E^j} \tilde g_j = \int_{E^j}   g_j$, we conclude 
\begin{align*}
\prod_{i=1}^k \left(   \int_{E_i} f_i \right)^{c_i} \leq   e^{D_g(\mathbf{c},\mathbf{d},\mathbf{B}_T) +D_g(\mathbf{c},\mathbf{d},\mathbf{B}_{E_0/  T}) } 
\prod_{j=1}^m \left(   \int_{E^j} g_j  \right)^{d_j}
\end{align*}
for any non-negative, bounded measurable functions  $(f_i)_{1\leq i\leq k}$ and $(g_j)_{1\leq j\leq m}$ satisfying \eqref{FRBLhypZ}.  It was already noted in the proof of  Theorems \ref{thm:FRBLgaussianExtiff} and \ref{thm:FRBLunderGaussExt} that  bounded functions saturate the definition of $D(\mathbf{c},\mathbf{d},\mathbf{B})$, so the first claim is proved. 

The statement for $D_g$ follows by an identical argument, considering only centered gaussian functions. In fact, it is even easier since there are no measurability considerations to deal with. 
\end{proof}

\begin{lemma}\label{lem:decomposabilityOfD}
Let $T\subset E_0$ be a critical subspace for the datum $(\mathbf{c},\mathbf{d},\mathbf{B})$. It holds that 
$$
D_g(\mathbf{c},\mathbf{d},\mathbf{B}) = D_g(\mathbf{c},\mathbf{d},\mathbf{B}_T) +D_g(\mathbf{c},\mathbf{d},\mathbf{B}_{E_0/  T}) .
$$
\end{lemma}
\begin{remark}
The same conclusion also holds for $D(\mathbf{c},\mathbf{d},\mathbf{B})$, though we do not need to prove it separately here.  It will follow from subsequent results, and is stated later as Corollary \ref{cor:decomposabilityOfDD}. 
\end{remark}
\begin{proof} We  take advantage of the entropic characterization of $D_g(\mathbf{c},\mathbf{d},\mathbf{B})$ in Proposition \ref{prop:entropyCharacterizationDg}  to give a simple proof, though it is also possible to appeal to the functional formulation.  We assume  $D_g(\mathbf{c},\mathbf{d},\mathbf{B})<+\infty$, since otherwise the corresponding claims follow from Lemma \ref{lem:Dsubadditive} and the fact that $D_g>-\infty$ for any datum.  

Recall that critical subspaces are of product form by definition.   Define $T_i := \pi_{E_i}T$, and let $X_i,Y_i$ be independent, gaussian random vectors in $T_i,T^{\perp}_i$, respectively (each having finite entropies by definition).  Define $Z_i = \epsilon^{-1} X_i+ Y_i$, which is a gaussian random vector in $E_i$, and note that
\begin{align*}
h(c_i^{-1} Z_i) = h(\epsilon^{-1} c_i^{-1}  X_i,c_i^{-1} Y_i) &= h(\epsilon^{-1} c_i^{-1} X_i)+h(c_i^{-1} Y_i) \\
&= \dim(T_i) \log(\epsilon^{-1}) + h(c_i^{-1} X_i)+h(c_i^{-1} Y_i), ~~1\leq i\leq k.
\end{align*}
Now, for any coupling of the $(Z_i)_{1\leq i\leq k}$ satisfying \eqref{couplingFiniteEntropy}, it follows  by subadditivity of entropy (Proposition \ref{app:Subadditive}) and the scaling property (Proposition \ref{app:scaling}) that
\begin{align}
&h\left(\sum_{i=1}^k   B_{ij}  Z_i  \right) \\
&= h\left(\epsilon^{-1} \sum_{i=1}^k   B_{ij}  X_i  +    \pi_{B_{j}T }  \sum_{i=1}^k  B_{ij}   Y_i ~  ,   ~\pi_{(B_{j}T)^{\perp}} \sum_{i=1}^k  B_{ij}   Y_i    \right) \\
&\leq h\left(\epsilon^{-1} \sum_{i=1}^k   B_{ij}  X_i  +    \pi_{B_{j}T } \sum_{i=1}^k  B_{ij}   Y_i    \right) + h\left(  \pi_{(B_{j}T)^{\perp}} \sum_{i=1}^k  B_{ij}   Y_i \right)\\
&= \dim(B_j T) \log(\epsilon^{-1})  + h\left(  \sum_{i=1}^k   B_{ij}  X_i  +    \epsilon ~\pi_{B_{j}T } \sum_{i=1}^k  B_{ij}   Y_i    \right) + h\left(  \pi_{(B_{j}T)^{\perp}} \sum_{i=1}^k  B_{ij}   Y_i \right).
\end{align}
Recalling Remark \ref{rmk:FiniteEntropies}, we note that all terms are finite.  

So, for any $\epsilon>0$, we use the assumption that $T$ was critical  to cancel the $\log(\epsilon^{-1})$ terms to find
\begin{align*}
D(\mathbf{c},\mathbf{d},\mathbf{B})
& \geq  
  \sum_{i=1}^k c_i h(c_i^{-1}X_i) -  \sum_{j=1}^m d_j  h\left(  \sum_{i=1}^k   B_{ij}  X_i  +    \epsilon ~\pi_{B_{j}T } \sum_{i=1}^k  B_{ij}   Y_i    \right)  \\
  &~~~~+  \sum_{i=1}^k c_i h( c_i^{-1} Y_i) -   \sum_{j=1}^m d_j   h\left(  \pi_{(B_{j}T)^{\perp}} \sum_{i=1}^k  B_{ij}   Y_i \right) 
\end{align*}
for some coupling of the $(Z_i)_{1\leq i\leq k} \equiv (X_i,Y_i)_{1\leq i\leq k}$.
Next, since gaussians have bounded second moments by definition,  weak upper semicontinuity of  entropy (Proposition \ref{app:WeakUSC})  implies
$$
 \limsup_{\epsilon\to 0}  h\left(  \sum_{i=1}^k   B_{ij}  X_i  +    \epsilon ~\pi_{B_{j}T } \sum_{i=1}^k  B_{ij}   Y_i    \right)  \leq  h\left(  \sum_{i=1}^k   B_{ij}  X_i   \right) , ~~~~1\leq i\leq k.
$$
In combination with the previous estimate, we have
\begin{align*}
D_g(\mathbf{c},\mathbf{d},\mathbf{B}) & \geq  
  \sum_{i=1}^k c_i h(c_i^{-1}X_i) -  \max_{P_{\mathbf{X}}\in \Pi(P_{X_1},\dots, P_{X_k})}  \sum_{j=1}^m d_j  h\left(  \sum_{i=1}^k   B_{ij,T_i}  X_i      \right) \\
  & ~~~~ +  \sum_{i=1}^k c_i h( c_i^{-1} Y_i) - \max_{P_{\mathbf{Y}}\in \Pi(P_{Y_1},\dots, P_{Y_k})} \sum_{j=1}^m d_j    h\left(  \sum_{i=1}^k   B_{ij,T_i^{\perp}}   Y_i \right) . 
\end{align*}
Since we chose  $(X_i)_{1\leq i\leq k}$ and $(Y_i)_{1\leq i\leq k}$ to be arbitrary gaussians on their respective subspaces,   it follows that
$$
D_g(\mathbf{c},\mathbf{d},\mathbf{B}) \geq D_g(\mathbf{c},\mathbf{d},\mathbf{B}_T) +D_g(\mathbf{c},\mathbf{d},\mathbf{B}_{E_0/  T}) .
$$
Comparing to  Lemma \ref{lem:Dsubadditive}, the claim is proved. 
\end{proof}

\subsection{Necessary conditions for finiteness}

\begin{proposition}\label{prop:necessaryCondtions}
If $D(\mathbf{c},\mathbf{d},\mathbf{B})<+\infty$ or $D_g(\mathbf{c},\mathbf{d},\mathbf{B})<+\infty$, then we must have \eqref{CONDITIONscaling} and 
\begin{align}
\sum_{i=1}^k c_i \dim (\pi_{E_i} T) \leq \sum_{j=1}^m d_j \dim (B_j T) ~~~\mbox{for all product-form subspaces $T\subseteq E_0$.} \label{subspaceNecessaryCondn}
\end{align}
In particular, each $B_j$ must be surjective. 
\end{proposition}
\begin{proof}
Since $D_g(\mathbf{c},\mathbf{d},\mathbf{B}) \leq D(\mathbf{c},\mathbf{d},\mathbf{B})$ by definition, it suffices to establish necessary conditions for  $D_g(\mathbf{c},\mathbf{d},\mathbf{B})$ to be finite. 
The condition \eqref{CONDITIONscaling} can be easily seen using the scaling property of entropy  (Proposition \ref{app:scaling}) by multiplying all random variables in \eqref{gaussianEntropyDg} by a common scalar factor.  

The necessity of \eqref{subspaceNecessaryCondn} follows immediately from the proof of Lemma \ref{lem:decomposabilityOfD}, but without cancelling the $\sum_{i=1}^k c_i \dim(T_i) \log(\epsilon^{-1})$ and $\sum_{j=1}^m d_j \dim(B_j T) \log(\epsilon^{-1})$ terms.  These terms  cancelled previously under the assumption that $T$ was critical, but this will not be the case if we assume $T$ is such that 
$$
 \sum_{i=1}^k c_i \dim (\pi_{E_i} T) > \sum_{j=1}^m d_j \dim (B_j T) ,
$$
leading to an arbitrarily large lower bound on  $D_g(\mathbf{c},\mathbf{d},\mathbf{B})$ as $\epsilon$ vanishes.

To see that each $B_j$ must be surjective, we take $T=E_0$ and compare \eqref{eq:consistency} to \eqref{subspaceNecessaryCondn}.  
\end{proof}

\subsection{Sufficient conditions for finiteness and gaussian-extremizability}
The goal of this section is to establish the sufficiency of the conditions in Theorem \ref{thm:necSuffCondDgFinite} for finiteness and gaussian-extremizability.   We start with a technical lemma, which is the counterpart of \cite[Lemma 5.1]{bennett2008brascamp} for our setting.  

\begin{lemma}\label{lem:FindBasis}
Define $N:= \dim(E_0)$ and let $(\mathbf{c},\mathbf{d},\mathbf{B})$ be a datum such that \eqref{CONDITIONscaling} holds and 
\begin{align}
\sum_{j=1}^m d_j \dim ( B_j T ) \geq \sum_{i=1}^k c_i \dim(\pi_{E_i} T)  ~~\mbox{for all product-form subspaces $T\subseteq E_0$.}\label{dimCondn}
\end{align}
In particular, this implies each $B_j$ is surjective.   Then, there is a real number $c>0$ such that, for every orthonormal basis $(e_n)_{1\leq n\leq N}$ of $E_0$ with the property that each $e_n \in E_i$ for some $1\leq i\leq k$, there exists a set $I_j \subseteq \{1,\dots, N\}$ for each $1\leq j\leq m$ with $|I_j|=\dim(E^j)$ such that
\begin{align}
\sum_{j=1}^m d_j |I_j \cap \{  n+1, \dots, N \}| \geq \sum_{i=1}^k c_i |S_i \cap \{n+1,\dots, N \}| ~~~\mbox{for all~}0\leq n\leq N, \label{techEstimate}
\end{align}
where $S_i : = \{ n : e_{n} \in E_i\}$, $1\leq i\leq k$ 
and
\begin{align}
\left\| \bigwedge_{n \in I_j }    B_j  e_n \right\|_{E_j}   \geq c~~~\mbox{for all~}1\leq j\leq m. \label{wedgeEstimate}
\end{align}
Moreover, if there are no critical subspaces, then there is a constant $\delta>0$ depending only on the datum $(\mathbf{c},\mathbf{d},\mathbf{B})$ such that  
\begin{align}
\sum_{j=1}^m d_j |I_j \cap \{  n+1, \dots, N \}| \geq \sum_{i=1}^k c_i |S_i \cap \{n+1,\dots, N \}| +\delta ~~~\mbox{for all~}0< n< N. \label{techEstimate2}
\end{align}
\end{lemma}
\begin{proof}
Since the space of all orthonormal bases is compact, and the number of possible $I_j$ is finite, it follows by continuity and compactness that \eqref{wedgeEstimate} may be replaced by the weaker assumption that $ \left(B_j e_n \right)_{n \in I_j}$ are linearly independent in $E_j$ for each $1\leq j\leq m$.

Now, we construct $I_j$ by a backwards greedy algorithm.  Specifically, we set $I_j$ equal to those indices $n$ for which $B_j e_n$ is not in the linear span of $\{B_j e_{n'}   ; n < n' \leq \dim(E_0) \}$. Since $B_j$ is surjective, we will have $|I_j|=\dim(E^j)$.  To prove \eqref{techEstimate}, we first fix $n$ satisfying $0 < n < N$, and apply \eqref{dimCondn} with $T$ equal to the span of $\{e_{n+1}, \dots, e_{\dim(E_0)}\}$, which is of product form by the assumption that each  $e_n \in \bigcup_{i=1}^k E_i$.  Specifically, due to construction of $I_j$ we have
$$
\dim ( B_j T) = |I_j \cap \{ n+1, \dots, N\}|.
$$
On the other hand, 
$$
\dim( \pi_{E_i} T) =   |S_i \cap \{n+1, \dots, N\}|,
$$
establishing \eqref{techEstimate} when $0< n< N$.  The case of $n=N$ is trivial, and the case of $n=0$ follows from  equality in \eqref{dimCondn} for $T=E_0$ since $|S_i\cap\{1,\dots, N\} = |S_i| = \dim(E_i)$, and $|I_j\cap\{1,\dots, N\} = |I_j| = \dim(E^j)$.  

Now, if there are no critical subspaces, then there is $\delta>0$ depending only on the datum $(\mathbf{c},\mathbf{d},\mathbf{B})$ such that \eqref{dimCondn} can be refined to 
\begin{align}
\sum_{j=1}^m d_j \dim ( B_j T ) \geq \sum_{i=1}^k c_i \dim(\pi_{E_i} T) + \delta
\end{align}
for all non-zero proper subspaces $T\subset E_0$.  Indeed, this easily follows since the left and right sides of \eqref{dimCondn} only take finitely many values.  Incorporating this into the previous analysis gives \eqref{techEstimate2}.
\end{proof}

\begin{proposition}[Sufficient conditions for finiteness and gaussian-extremizability]\label{prop:sufficientCondtions}
If the datum $(\mathbf{c},\mathbf{d},\mathbf{B})$  is such that \eqref{CONDITIONscaling} and \eqref{CONDITIONdimension} hold, then $D_g(\mathbf{c},\mathbf{d},\mathbf{B})$ is finite.  If it further holds that $(\mathbf{c},\mathbf{d},\mathbf{B})$ is simple, then $(\mathbf{c},\mathbf{d},\mathbf{B})$ is gaussian-extremizable. 
\end{proposition}
\begin{proof}The argument  follows the  strategy of proof for \cite[Proposition 5.2]{bennett2008brascamp}, but is recast in terms of entropies which we find more convenient in the present setting\footnote{Note that we work exclusively with gaussian random vectors here, so all computations can be stated in terms of determinants, using the identity \eqref{gaussEntropyExpression}.}. 
Define $N:= \dim(E_0)$ and consider gaussian random vectors $Z_i$ in $E_i$, $1\leq i\leq k$. It is trivially true that 
\begin{align}
\max_{P_{\mathbf{Z}}\in \Pi(P_{Z_1},\dots, P_{Z_k})} \sum_{j=1}^m d_j h\left(\sum_{i=1}^k c_i B_{ij} Z_i\right) \geq \sum_{j=1}^m d_j h\left(\sum_{i=1}^k c_i B_{ij} Z'_i\right), \label{eq:boundMaxEntropyByIndep}
\end{align}
where $Z'_i=Z_i$ in distribution for each $1\leq i\leq k$, and $Z'_1, \dots, Z'_k$ are independent. Without loss of generality, we may write
$$
Z'_i = \sum_{n\in S_i} W_n e_n,~~~1\leq i\leq k
$$
where $(e_n)_{n\in S_i} \subset E_i\subset E_0$ is an orthonormal basis for $E_i$, $1\leq i\leq k$, $(S_i)_{1\leq i\leq k}$ is a partition of $\{1,\dots, N\}$, and   $(W_n)_{1\leq n \leq N}$ is a collection of independent one-dimensional gaussian random variables.  We may further assume that the indices are chosen to satisfy $h(W_1)  \leq \cdots \leq h(W_N)$. 
Now, we invoke Lemma \ref{lem:FindBasis} and it follows from the scaling property for entropy (Proposition \ref{app:scaling}) that 
\begin{align}
h\left(\sum_{i=1}^k  B_{ij} Z'_i\right) =
h\left(\sum_{n=1}^N W_n B_j e_n \right)
 \geq h\left(\sum_{n\in I_j} W_n B_j e_n  \right)
\geq h((W_n)_{n\in I_j}) + C,
\end{align}
for some constant $C$ depending only on the datum $(\mathbf{c},\mathbf{d},\mathbf{B})$ since $(B_j e_n)_{n\in I_j}$ form a basis of $E^j$ with a lower bound on degeneracy.

Now,  by telescoping and Lemma \ref{lem:FindBasis}, we may write
\begin{align*}
&\sum_{j=1}^m d_j h((W_n)_{n\in I_j}) \\
&= \sum_{n=1}^N h(W_n) \sum_{j=1}^m d_j |I_j \cap \{n\}|\\
&=\left( \sum_{j=1}^m d_j |I_j| \right) h(W_1) + \sum_{n=1}^{N-1} \left( h(W_{n+1})-h(W_n) \right) \sum_{j=1}^m d_j |I_j \cap \{n+1,\dots, N\}|\\
&=\left( \sum_{i=1}^k c_i |S_i| \right) h(W_1) + \sum_{n=1}^{N-1} \left( h(W_{n+1})-h(W_n) \right) \sum_{j=1}^m d_j |I_j \cap \{n+1,\dots, N\}|\\
&\geq \left( \sum_{i=1}^k c_i |S_i| \right) h(W_1) + \sum_{n=1}^{N-1} \left( h(W_{n+1})-h(W_n) \right) \left( \sum_{i=1}^k c_i |S_i  \cap \{n+1,\dots, N\}|+ \delta\right) \\
&= \sum_{n=1}^{N} h(W_n)  \sum_{i=1}^k c_i |S_i  \cap \{n\}| + \delta \sum_{n=1}^{N-1} \left( h(W_{n+1})-h(W_n) \right)\\
&=\sum_{i=1}^k c_i \left(\sum_{n\in S_i}h(W_n) \right)+ \delta  \left( h(W_{N})-h(W_1) \right)\\
&=\sum_{i=1}^k c_i h(Z_i) + \delta  \left( h(W_{N})-h(W_1) \right).
\end{align*}
So, we conclude 
\begin{align}
 \sum_{i=1}^k c_i h(c_i^{-1} Z_i) - \max_{P_{\mathbf{Z}}\in \Pi(P_{Z_1}, \dots, P_{Z_k})} \sum_{j=1}^m d_j h\left(\sum_{i=1}^k B_{ij} Z_i\right) \leq C' - \delta  \left( h(W_{N})-h(W_1) \right)  \label{tempZineq}
\end{align}
for constants $C',\delta\geq 0$ depending only on the datum $(\mathbf{c},\mathbf{d},\mathbf{B})$.  In particular, $D_g(\mathbf{c},\mathbf{d},\mathbf{B})$ is finite.  

Now, if $(\mathbf{c},\mathbf{d},\mathbf{B})$ is simple, then the last claim of  Lemma \ref{lem:FindBasis} implies  $\delta>0$.  Since the LHS of \eqref{tempZineq} is invariant to scaling each $Z_i$ by a common factor (due to the scaling condition \eqref{CONDITIONscaling}), it easily follows that there are constants $c_1,c_2>0$ depending only on the datum $(\mathbf{c},\mathbf{d},\mathbf{B})$ such that we may restrict our attention to $Z_i$ satisfying 
$$
c_1 \leq \Var(W_1) \leq \lambda_{\min}(\Cov(Z_i)) \leq \lambda_{\max}(\Cov(Z_i)) \leq \Var(W_N)  \leq c_2~~~\mbox{for each $1\leq i\leq k$.}
$$
Thus, in supremizing the LHS of  \eqref{gaussianEntropyDg}, it suffices to consider gaussian $Z_i$ with covariances in a compact set, with eigenvalues uniformly bounded away from zero. Equivalently, in supremizing the functional 
$$
 (K_1, \dots, K_k) \longmapsto \left( \sum_{i=1}^k c_i \log \det K_i -  \max_{K\in \Pi(K_1, \dots, K_k)} \sum_{j=1}^m d_j \log \det (B_j \Lc K \Lc B_j^*)\right)
$$
over $\prod_{i=1}^k S^+(E_i)$, it suffices to consider each $K_i$ in a compact set, with eigenvalues bounded away from zero.  It therefore follows by upper-semicontinuity (i.e., Lemma \ref{lem:semicontinuityOfgaussian}) that an extremizer exists.  Thus, $(\mathbf{c},\mathbf{d},\mathbf{B})$ is gaussian-extremizable. 
\end{proof}

\begin{proof}[Proof of Theorem \ref{thm:necSuffCondDgFinite}]
The claim is an immediate corollary of Propositions \ref{prop:necessaryCondtions} and \ref{prop:sufficientCondtions}.
\end{proof}

We may now also prove Theorem \ref{thm:FRBL}, which proceeds just as in \cite[Proof of Theorems 1.9 and 1.15]{bennett2008brascamp}:
\begin{proof}[Proof of Theorem \ref{thm:FRBL}]
In view of Theorem \ref{thm:FRBLdualityConstants}, we only need to prove that $D(\mathbf{c},\mathbf{d},\mathbf{B})=D_g(\mathbf{c},\mathbf{d},\mathbf{B})$. To do this, we induct on the dimension $\dim(E_0)$.  The case $\dim(E_0)=0$ is trivial, so assume the claim holds for smaller values of $\dim(E_0)$.

We may assume $D_g(\mathbf{c},\mathbf{d},\mathbf{B})<+\infty$, else the claim is trivial since $D(\mathbf{c},\mathbf{d},\mathbf{B})\geq D_g(\mathbf{c},\mathbf{d},\mathbf{B})$ by definition.  Thus, we assume that \eqref{CONDITIONscaling} and \eqref{CONDITIONdimension} hold, since these are necessary conditions for finiteness of $D_g(\mathbf{c},\mathbf{d},\mathbf{B})$ by Theorem \ref{thm:necSuffCondDgFinite}.  If $(\mathbf{c},\mathbf{d},\mathbf{B})$ is simple, then it is also gaussian-extremizable by Theorem \ref{thm:necSuffCondDgFinite}, so the desired claim follows by Theorem \ref{thm:FRBLunderGaussExt}.   On the other hand, if $(\mathbf{c},\mathbf{d},\mathbf{B})$ is not simple, then by Lemma \ref{lem:decomposabilityOfD} and the definition of simple, there exists a critical subspace $T\subset E_0$ for which 
$$
D_g(\mathbf{c},\mathbf{d},\mathbf{B}) = D_g(\mathbf{c},\mathbf{d},\mathbf{B}_T) +D_g(\mathbf{c},\mathbf{d},\mathbf{B}_{E_0/  T}) .
$$
By Lemma \ref{lem:Dsubadditive}, we also have
$$
D(\mathbf{c},\mathbf{d},\mathbf{B}) \leq D(\mathbf{c},\mathbf{d},\mathbf{B}_T) +D(\mathbf{c},\mathbf{d},\mathbf{B}_{E_0/  T}). 
$$
By the induction hypothesis, 
$$
D(\mathbf{c},\mathbf{d},\mathbf{B}_T) = D_g(\mathbf{c},\mathbf{d},\mathbf{B}_T)
$$
and 
$$
D(\mathbf{c},\mathbf{d},\mathbf{B}_{E_0/  T})=D_g(\mathbf{c},\mathbf{d},\mathbf{B}_{E_0/  T}).
$$
Combining the above estimates, we have
$$
D(\mathbf{c},\mathbf{d},\mathbf{B}) \leq D(\mathbf{c},\mathbf{d},\mathbf{B}_T) +D(\mathbf{c},\mathbf{d},\mathbf{B}_{E_0/  T}) = D_g(\mathbf{c},\mathbf{d},\mathbf{B}_T) +D_g(\mathbf{c},\mathbf{d},\mathbf{B}_{E_0/  T}) = D_g(\mathbf{c},\mathbf{d},\mathbf{B}). 
$$
Taken together with the trivial inequality $D_g(\mathbf{c},\mathbf{d},\mathbf{B}) \leq D(\mathbf{c},\mathbf{d},\mathbf{B})$, we must have equality.  This closes the induction and completes the proof. 
\end{proof}
In analogy to Lemma \ref{lem:decomposabilityOfD}, the following corollary is now immediate.  It is not needed elsewhere, but we state it for completeness. 
\begin{corollary}\label{cor:decomposabilityOfDD}
Let $T\subset E_0$ be a critical subspace for the datum $(\mathbf{c},\mathbf{d},\mathbf{B})$. It holds that 
$$
D(\mathbf{c},\mathbf{d},\mathbf{B}) = D(\mathbf{c},\mathbf{d},\mathbf{B}_T) +D(\mathbf{c},\mathbf{d},\mathbf{B}_{E_0/  T}) .
$$
\end{corollary}

\section{Connections to other Brascamp-Lieb-type inequalities} \label{sec:connections}
\subsection{The Brascamp-Lieb inequality} \label{sec:connectionsBL}
It is clear by now that the Brascamp-Lieb inequality is a special case of the forward-reverse inequality.  On the other hand, if we assume  
  \eqref{CONDITIONscaling} holds and that \eqref{CONDITIONdimension}  holds for \emph{all} subspaces, not just those of product form, then finiteness of $D(\mathbf{c},\mathbf{d},\mathbf{B})$ be established as a consequence of the finiteness conditions for the forward Brascamp-Lieb inequality.  The argument is as follows, and  is due to Michael Christ.  

Assume \eqref{CONDITIONscaling} and further assume that \eqref{CONDITIONdimension} holds for all subspaces $T$ (not just those of product form). Define the index sets $I :=\{1,\dots, k\}$ and $J = \{1,\dots,m\}$.  Since the statement and conclusion are invariant to rescaling $\mathbf{c},\mathbf{d}$ by the same constant, we assume without loss of generality that $\max_{i\in I}c_i < 1$   and $\max_{j\in J}d_j < 1$.  Now, assuming $I,J$ are disjoint index sets, we define the augmented index set $J^{\star} = I \cup J$.  For $j \in J^{\star}\setminus J$, define $d_j =(1-c_i)$, $E^j = E_i$, and $B_j = \pi_{E_i}\Lc^{-1}$.   Now, if $(f_i)_{i\in I}$ and $(g_j)_{j\in J}$ satisfy \eqref{FRBLhypIntro}, then defining $g_j = f_i$ for $j \in J^{\star}\setminus J$, it follows that 
$$
\prod_{i\in I} f_i(\pi_{E_i}(x))  \leq \prod_{j\in J} g^{d_j}_j (B_j \Lc x) \prod_{i\in I} f^{1-c_i}_i (\pi_{E_i} x) =   \prod_{j\in J^{\star}} g^{d_j}_j (B_j \Lc x)  .
$$
Integrating over both sides and using the fact that $E_0 = \bigoplus_{i=1}^k E_i$, we obtain
\begin{align}
 \prod_{i\in I} \int_{E_i} f_i  \leq \int_{E_0}   \prod_{i\in I} f_i(\pi_{E_i}(x)) dx    \leq  \int_{E_0} \prod_{j\in J^{\star}} g^{d_j}_j (B_j \Lc x)  dx .\label{IntegrateBothSides1}
\end{align}
By the finiteness criteria for the forward Brascamp-Lieb inequality \cite[Theorem 1.13]{bennett2008brascamp}, 
\begin{align}
\int_{E_0} \prod_{j\in J^{\star}} g^{d_j}_j (B_j \Lc x)  dx \leq e^D \prod_{j\in J^{\star}}\left( \int_{E^j} g_j \right)^{d_j}= e^D \prod_{j\in J }\left( \int_{E^j} g_j \right)^{d_j}\prod_{i\in I }\left( \int_{E_i} f_i \right)^{1-c_i}, \label{applyFBL}
\end{align}
where $D<+\infty$ provided
\begin{align}
\sum_{j\in J^{\star}} d_j \dim ( B_j T ) \geq  \dim( T) ~~\mbox{for all subspaces $T\subseteq E_0$, }\label{dimCondnFBL}
\end{align}
and further holding with equality when $T=E_0$.  Assuming this is true for the moment, we combine \eqref{IntegrateBothSides1} and \eqref{applyFBL} to conclude $D(\mathbf{c},\mathbf{d},\mathbf{B})\leq D <+\infty$, as desired. 

So, to verify \eqref{dimCondnFBL} and therefore justify the application \eqref{applyFBL}, observe that, since we assumed \eqref{CONDITIONdimension} for all subspaces $T\subseteq E_0$,
\begin{align}
\sum_{j\in J^{\star}} d_j \dim ( B_j T ) &= \sum_{j\in J} d_j \dim ( B_j T )  + \sum_{i\in I} (1-c_i) \dim ( \pi_{E_i} T ) \\
&\geq \sum_{i\in I}   \dim ( \pi_{E_i} T )   \geq \dim(T),
\end{align}
with equality holding when $T=E_0$ by \eqref{CONDITIONscaling}.

\begin{remark}
The disadvantage of the above argument is that it will not, in general, recover the sharp constant (and therefore the gaussian saturation property) for the forward-reverse Brascamp-Lieb inequality, even under the stronger condition that \eqref{CONDITIONdimension} holds for all subspaces $T\subseteq E_0$.  
\end{remark}

\begin{remark}
Example \ref{ex:RevYoung} (reverse Young inequality) provides an important counterpoint to the above discussion.  The reader can check that \eqref{CONDITIONdimension} is verified for all product-form subspaces, however it fails to hold for some non-product form subspaces.  Hence,  bootstrapping   the direct Brascamp-Lieb inequality as above would fail to give a finite constant in the reverse Young inequality. 
\end{remark}

\subsection{The Barthe-Wolff inverse Brascamp-Lieb inequality} \label{subsec:BartheWolff}

The following  ``inverse" Brascamp-Lieb inequality was announced by Barthe and Wolff in the note \cite{barthe2014positivity} and proved rigorously in \cite{barthe2018positive}. We write it in a form  to emphasize the connection to Theorem \ref{thm:FRBL}.
\begin{theorem} \label{thm:BartheWolff}
Let $C \in (-\infty,+\infty]$ be any constant, and let previously introduced notation prevail.  For any measurable functions   $f_i : E_i \to \R^+$,  $1\leq i \leq k$ and $g_j : E^j \to \R^+$,  $1\leq j\leq m$, 
\begin{align}
 \prod_{i=1}^k \left( \int_{E_i} f_i \right)^{c_i}  \prod_{j=1}^m \left( \int_{E^j} g_j \right)^{-d_j}    \leq e^C \int_{E_0} \prod_{i=1}^k f_i^{c_i}(\pi_{E_i} x)   \prod_{j=1}^m g_j^{-d_j}(B_j x)  dx  \label{BartheWolff}
\end{align}
if and only if \eqref{BartheWolff} holds for all centered gaussian functions $(f_i)_{1\leq i\leq k}$ and $(g_j)_{1\leq j\leq m}$. 
\end{theorem}

For sake of comparison, we restate the gaussian saturation part of Theorem \ref{thm:FRBL} here in equivalent form as follows:
\begin{theorem} \label{thm:FRBLBartheWolff}
Let $D \in (-\infty,+\infty]$ be any constant, and let previously introduced notation prevail.  For any measurable functions   $f_i : E_i \to \R^+$,  $1\leq i \leq k$ and $g_j : E^j \to \R^+$,  $1\leq j\leq m$ satisfying
\begin{align}
\prod_{i=1}^k f_i^{c_i}(\pi_{E_i} x) \leq \prod_{j=1}^m g_j^{d_j}\left( B_j x \right)\hspace{1cm}\forall x \in E_0 ,\label{FRBLhypBW}
\end{align}
we have  
\begin{align}
\prod_{i=1}^k \left( \int_{E_i} f_i  \right)^{c_i} \leq e^D \prod_{j=1}^m \left( \int_{E^j} g_j  \right)^{d_j},\label{FRBL_BW}
\end{align}
if and only if \eqref{FRBL_BW} holds for all centered gaussian functions $(f_i)_{1\leq i\leq k}$ and $(g_j)_{1\leq j\leq m}$ satisfying \eqref{FRBLhypBW}.
\end{theorem}

To see the connection between the two results, we first note that  Theorem \ref{thm:FRBLBartheWolff} implies Theorem \ref{thm:BartheWolff} by augmenting the datum $(\mathbf{c},\mathbf{d},\mathbf{B})$ with $E^{m+1}=E_0$, $d_{m+1}=1$ and $B_{m+1} = \operatorname{id}_{E_0}$.   By choosing the function  $g_{m+1} : E^{m+1}\to \R^+$ according to 
$$
g_{m+1}(x) =  \prod_{i=1}^k f_i^{c_i}(\pi_{E_i} x)   \prod_{j=1}^m g_j^{-d_j}(B_j x),  
$$
the hypothesis \eqref{FRBLhypBW} is  satisfied, and therefore \eqref{BartheWolff} follows from \eqref{FRBL_BW}.  For given functions $(f_i)_{1\leq i\leq k}$ and $(g_j)_{1\leq j\leq m}$, the above choice of $g_{m+1}$ is clearly best-possible, so the best constant $C$ in \eqref{BartheWolff} must be equal to the best constant $D$ in Theorem \ref{thm:FRBLBartheWolff} for the augmented datum, which can be computed by considering only centered gaussian functions.

In fact, the reverse is also true.  That is, Theorem \ref{thm:FRBLBartheWolff} may be derived from Theorem \ref{thm:BartheWolff}.  The argument is a bit   less straightforward in comparison, but nevertheless  brief.  The idea is to apply Theorem \ref{thm:BartheWolff} with exponents $c'_i = 1+t c_i$ and $d'_j = t d_j$, where $t>0$ is a parameter that will tend to $+\infty$.    For this choice of exponents, we apply the pointwise inequality \eqref{FRBLhypBW} to see that the RHS of \eqref{BartheWolff} can be upper bounded as
$$
 \int_{E_0} \prod_{i=1}^k f_i^{c'_i}(\pi_{E_i} x)   \prod_{j=1}^m g_j^{-d'_j}(B_j x)  dx \leq  \int_{E_0} \left( \prod_{i=1}^k f_i^{c'_i}(\pi_{E_i} x)  \right) \left( \prod_{i=1}^k f_i^{-tc_i}(\pi_{E_i} x)  \right) dx =  \prod_{i=1}^k  \left( \int_{E_i} f_i \right).
$$
Invoking \eqref{BartheWolff} itself and dividing exponents by $t$, we find that \eqref{FRBLhypBW} implies
\begin{align}
 \prod_{i=1}^k \left( \int_{E_i} f_i \right)^{c_i}   \leq e^{C_t/t}  \prod_{j=1}^m \left( \int_{E^j} g_j \right)^{d_j}   , \notag
 \end{align}
 where $C_t$ denotes the best constant in the inequality \eqref{BartheWolff} for the exponents $(c'_i)_{1\leq i\leq k}$ and $(d'_j)_{1\leq j\leq m}$.  In particular, for $D$ the best constant in \eqref{FRBL_BW}, we have $D \leq C_t/t$ for all $t>0$.  By the gaussian saturation claim of Theorem \ref{thm:BartheWolff} and direct computation (see \cite[Section 2.2]{barthe2018positive}), one may calculate 
 \begin{align}
 e^{2 C_t} = \sup \frac{\det\left( \sum_{i=1}^k c'_i \pi^*_{E_i}C_i   \pi_{E_i}  -  \sum_{j=1}^m d'_j B^*_j A_j  B_j    \right) }{  \prod_{i=1}^k (\det C_i )^{c'_i}  \prod_{j=1}^m (\det A_j )^{-d'_j}   },
 \end{align}
  where the supremum is over all $C_i \in S^+(E_i)$ and $A_j \in S^+(E^j)$ satisfying 
  \begin{align}
   \sum_{i=1}^k c'_i \pi^*_{E_i}C_i   \pi_{E_i}  \geq  \sum_{j=1}^m d'_j B^*_j A_j  B_j    . \label{matrixIneq}
  \end{align}
  The set of $(C_i)_{1\leq i\leq k}$ and $(A_j)_{1\leq j\leq m}$ satisfying \eqref{matrixIneq} are monotone decreasing in $t$ (with respect to inclusion), so in calculating $\liminf_{t\to \infty}C_t/t$, we need only consider positive-definite $(C_i)_{1\leq i\leq k}$ and $(A_j)_{1\leq j\leq m}$ in the intersection of all such sets; i.e., those satisfying 
  \begin{align}
   \sum_{i=1}^k c_i \pi^*_{E_i}C_i   \pi_{E_i}  \geq  \sum_{j=1}^m d_j B^*_j A_j  B_j    . \label{matrixIneq2}
  \end{align}
Assuming \eqref{matrixIneq2} holds, we bound
 \begin{align*}
 \frac{\det\left( \sum_{i=1}^k c'_i \pi^*_{E_i}C_i   \pi_{E_i}  -  \sum_{j=1}^m d'_j B^*_j A_j  B_j    \right)^{1/t} }{  \prod_{i=1}^k (\det C_i )^{1/t+c_i}  \prod_{j=1}^m (\det A_j )^{-d_j}   } &\leq  \frac{\det\left( \sum_{i=1}^k c'_i \pi^*_{E_i}C_i   \pi_{E_i}    \right)^{1/t} }{  \prod_{i=1}^k (\det C_i )^{1/t+c_i}  \prod_{j=1}^m (\det A_j )^{-d_j}   } \\
 &=\frac{ \prod_{i=1}^k (1+t c_i)^{\dim(E_i)/t}  }{  \prod_{i=1}^k (\det C_i )^{ c_i}  \prod_{j=1}^m (\det A_j )^{-d_j}   } .
 \end{align*}
Hence, 
$$
\liminf_{t\to \infty} C_t/t \leq \sup \left( \frac{1}{2}\sum_{j=1}^{m} d_j \log\det A_j -\frac{1}{2}\sum_{i=1}^{k} c_i \log\det C_i \right), $$
where the supremum is over all $C_i \in S^+(E_i)$ and $A_j \in S^+(E^j)$ satisfying  \eqref{matrixIneq2}.  This is precisely the best constant $D$ obtained in \eqref{FRBL_BW} by considering only centered gaussian functions, so the proof is complete. 

\begin{remark}
The above  argument showing equivalence of Theorems  \ref{thm:BartheWolff} and  \ref{thm:FRBLBartheWolff}  is due to Pawe\l{} Wolff.  Despite their formal equivalence, both results have their merits, and the techniques used to derive each are complimentary.  In particular, Barthe and Wolff use an optimal transport argument, while our proof relies primarily on duality and structural decomposition.   Additionally, the different formulations of the results have their respective advantages.  For example, Theorem \ref{thm:FRBL} highlights the unification of the Brascamp-Lieb and Barthe inequalities, together with the duality of best constants \eqref{eq:DualIdentity}.  On the other hand, Barthe and Wolff's formulation emphasizes an inverse principle to Lieb's \cite{lieb1990gaussian}.    Our proof is perhaps simpler since it avoids the detailed case analysis encountered in \cite{barthe2018positive}, but preference may depend on the reader's taste.
\end{remark}

\begin{remark}
Theorem \ref{thm:BartheWolff} is a particular case of the general inverse inequality by Barthe and Wolff which allows for integration against a nontrivial gaussian kernel in the RHS of \eqref{BartheWolff}, and for which the gaussian saturation property remains valid.  As we will see in the next section, their geometric inequality can be recovered in full generality with nontrivial gaussian kernel as a consequence of the geometric forward-reverse Brascamp-Lieb inequality.  Hence, there is a formal equivalence between the geometric Barthe-Wolff inequalities stated with (i) trivial gaussian kernel; or (ii) nontrivial gaussian kernel.  Although we do not pursue it here, it is an interesting question whether this formal equivalence continues to hold for  non-geometric instances of the Barthe-Wolff inequality. 
\end{remark}

\subsection{Inequalities with Gaussian kernels}\label{sec:gaussianKernels}

In this section,  we establish  inequalities for integrals against gaussian kernels as applications of our main results.  They are  easy corollaries of the geometric forward-reverse Brascamp-Lieb inequality.  Similar results could be stated for the general forward-reverse inequality, but we restrict attention to the geometric case to simplify the discussion.  %

\begin{definition}
For a Euclidean space $E$, we let $\gamma_{E}$ denote the standard gaussian measure on $E$.  That is,
$$
d\gamma_E(x) = \frac{1}{(2\pi)^{\dim(E)/2}}e^{-\frac{1}{2}|x|^2 }dx.
$$
\end{definition}

\begin{theorem}\label{thm:gaussianGeometricFRBL}
Let $H$ be a Euclidean space,   and $Q\in S(H)$ with signature $(s^+(Q),s^-(Q))$.  Consider linear maps $U_i : H\to E_i$ and $V_j : H \to E^j$  satisfying   $U_i U_i^* = \operatorname{id}_{E_i}$ and  $V_j V_j^* = \operatorname{id}_{E^j}$, for all $1\leq i \leq k$ and $1\leq j\leq m$.   Let $(c_i)_{1\leq i\leq k}$ and $(d_j)_{1\leq j\leq m}$ be positive numbers, and suppose that
\begin{align}
Q + \sum_{i=1}^k c_i U_i^* U_i =  \sum_{j=1}^m d_j V_j^* V_j > 0, \mbox{~~~and~~~}\dim(H) \geq s^+(Q) + \sum_{i=1}^k \dim(E_i).\label{gaussKernHyp1}
\end{align}
 If   measurable functions $f_i : E_i \to \R^+$,  $1\leq i \leq k$ and $g_j : E^j \to \R^+$,  $1\leq j\leq m$ satisfy
\begin{align}
\prod_{i=1}^k f_i^{c_i}(U_i x) \leq \prod_{j=1}^m g_j^{d_j}(V_j x)\hspace{1cm}\forall x\in H, \label{gaussKernelHyp}
\end{align}
then
\begin{align}
\prod_{i=1}^k \left(\int_{E_i} f_i d\gamma_{E_i} \right)^{c_i}
\leq  \prod_{j=1}^m \left(\int_{E^j} g_j d\gamma_{E^j} \right)^{d_j} . \label{gaussKernIneq}
\end{align}
\end{theorem}
\begin{proof}
Decompose $Q = Q^+ - Q^-$, with $Q^+,Q^- \in S^+(H)$.  By spectral decomposition,  write
$$
Q^+ = \sum_{\ell=1}^{s^+(Q)} \lambda_{\ell} u^*_{\ell} u_{\ell}; \hspace{1cm} Q^- = \sum_{\ell=1}^{s^-(Q)} \mu_{\ell} v^*_{\ell} v_{\ell},
$$
where $\lambda_{\ell}>0$ (resp. $\mu_{\ell}>0$) and $u_{\ell} u^*_{\ell} = \operatorname{id}_{\R}$ (resp. $v_{\ell} v^*_{\ell} = \operatorname{id}_{\R}$).  Thus, the first assumption in \eqref{gaussKernHyp1} can be written as
\begin{align}
\sum_{i=1}^k c_i U_i^* U_i  + \sum_{\ell=1}^{s^+(Q)} \lambda_{\ell} u^*_{\ell} u_{\ell}  =  \sum_{j=1}^m d_j V_j^* V_j  + \sum_{\ell=1}^{s^-(Q)} \mu_{\ell} v^*_{\ell} v_{\ell} >0 \label{eigenIdent}
\end{align}
In particular the LHS is a linear map with rank equal to $\dim(H)$ by the positive-definiteness assumption, so by subadditivity of rank, it holds that $\dim(H) \leq s^+(Q) + \sum_{i=1}^k \dim(E_i)$.  By \eqref{gaussKernHyp1}, we must have equality.  Thus, we can consider the map
$$
x\in H \to (U_1 x, \dots, U_k x, u_1 x, \dots, u_{s^+(Q)}x)
$$
as a bijective linear map from $H$ to $H$.  Now, if \eqref{gaussKernelHyp} is satisfied, then \eqref{eigenIdent} implies that we also have
\begin{align}
&\prod_{i=1}^k \left( f_i(U_i x) \frac{1}{(2\pi)^{\dim(E_i)/2}}e^{-\frac{1}{2} | U_i x|^2 }\right)^{c_i} \times  \prod_{\ell=1}^{s^+(Q)} \phi^{\lambda_{\ell}}(u_{\ell} x)  \notag\\
&\leq  \prod_{j=1}^m \left(  g_j(V_j x) \frac{1}{(2\pi)^{ \dim(E^j)/2}} e^{-\frac{1}{2} |V_j x|^2 }\right)^{d_j} \times  \prod_{\ell=1}^{s^-(Q)} \phi^{\mu_{\ell}}(v_{\ell} x), \hspace{1cm}\forall x\in H, \notag
\end{align}
where $\phi$ denotes the standard gaussian density
$$
\phi(z) := \frac{1}{2\pi}e^{-\frac{1}{2}| z|^2 }, ~~~z\in \R.
$$
Since $\int_{\R} \phi  = 1$, the inequality \eqref{gaussKernIneq} follows from an application of  Corollary \ref{cor:geometric2}.
\end{proof}

An important consequence of Theorem \ref{thm:gaussianGeometricFRBL} is the following geometric inverse Brascamp-Lieb inequality proved by Barthe and Wolff \cite[Theorem 4.7]{barthe2018positive}, recovered here in full generality.  We remark that the reverse H\"older-type inequality for gaussian random vectors due to Chen, Dafnis and Paouris \cite[Theorem 1(ii)]{chen2015improved} follows as a direct consequence  \cite[Section 4.3]{barthe2018positive}, so should be considered as yet another example. The direct Chen-Dafnis-Paouris inequality \cite[Theorem 1(i)]{chen2015improved} is a consequence of the forward Brascamp-Lieb inequality.
\begin{corollary}
Let $H$ be a Euclidean space, and consider linear maps $U_i : H\to E_i$ and $V_j : H \to E^j$  satisfying   $U_i U_i^* = \operatorname{id}_{E_i}$ and  $V_j V_j^* = \operatorname{id}_{E^j}$, for all $1\leq i \leq k$ and $1\leq j\leq m$.   Let $(c_i)_{1\leq i\leq k}$ and $(d_j)_{1\leq j\leq m}$ be positive numbers, and suppose that
\begin{align*}
Q + \sum_{i=1}^k c_i U_i^* U_i =  \sum_{j=1}^m d_j V_j^* V_j + \operatorname{id}_H, \mbox{~~~and~~~}\dim(H) \geq s^+(Q) + \sum_{i=1}^k \dim(E_i) 
\end{align*}
for $Q: H \to H$  a symmetric operator.
For all non-negative measurable functions $f_i : E_i \to \R^+$,  $1\leq i \leq k$ and $g_j : E^j \to \R^+$,  $1\leq j\leq m$, it holds that
\begin{align*}
\prod_{i=1}^k \left(\int_{E_i} f_i   \right)^{c_i} \prod_{j=1}^m \left(\int_{E^j} g_j   \right)^{-d_j} 
\leq \int_{H} e^{-\pi \langle Q x,x\rangle }  \prod_{i=1}^k f^{c_i}_i(U_i x)   \prod_{j=1}^m g^{-d_j}_j(V_j x)   dx.
\end{align*}
\end{corollary}
\begin{proof}
Define $E^{m+1}=H$, $V_{m+1} = \operatorname{id}_H$ and $d_{m+1}=1$.  Put
$$
\tilde{f}_i(x) := f_i\left((2\pi)^{-1/2}x\right)e^{\frac{1}{2}|x|^2}, 1\leq i\leq k; ~~\tilde{g}_j(x) := g_j\left((2\pi)^{-1/2}x\right)e^{\frac{1}{2}|x|^2},  1\leq j\leq m
$$
and define $\tilde{g}_{m+1} : H\to H$ defined according to 
\begin{align*}
\tilde{g}_{m+1} (x)  &:=    \prod_{i=1}^k \tilde{f}^{c_i}_i((2\pi)^{-1/2} U_i x)   \prod_{j=1}^m \tilde{g}^{-d_j}_j((2\pi)^{-1/2}V_j x) \\
&= \left( \prod_{i=1}^k  {f}^{c_i}_i((2\pi)^{-1/2} U_i x)   \prod_{j=1}^m  {g}^{-d_j}_j((2\pi)^{-1/2} V_j x) \right)  \exp\left(\frac{1}{2}|x|^2 - \frac{1}{2}\langle Q x,x\rangle\right).
 \end{align*} 
Now, by change of variables $u\leftarrow (2\pi)^{-1/2}x$, 
$$
\int_{E_i} \tilde f_i(x) d\gamma_{E_i}(x)  = \int_{E_i} f_i(u) du, ~1\leq i\leq k; ~~\int_{E^j} \tilde g_j(x) d\gamma_{E^j}(x) = \int_{E^j}  g_j(u) du,  ~1\leq j\leq m
$$
and
$$
\int_{H} \tilde g_{m+1}(x) d\gamma_{H}(x) = \int_{H} e^{-\pi \langle Q u,u\rangle }  \prod_{i=1}^k f^{c_i}_i(U_i u)   \prod_{j=1}^m g^{-d_j}_j(V_j u)   du.
$$
So, the claim follows from  Theorem  \ref{thm:gaussianGeometricFRBL}.
\end{proof}

\begin{remark}
We have seen above that the geometric Barthe-Wolff inequality  \cite[Theorem 4.7]{barthe2018positive} follows as a consequence of Corollary \ref{cor:geometric2}, the latter being a special case of the complete characterization of geometric instances of the Forward-Reverse Brascamp-Lieb inequality given in Corollary \ref{cor:geometric}.  So, it appears \emph{prima facie} that the class of geometric Forward-Reverse Brascamp-Lieb inequalities is more extensive than the geometric instances of the Barthe-Wolff inequality. 
\end{remark}

\subsection{The Anantharam-Jog-Nair inequality}\label{subsec:AJN}

In a recent paper \cite{anantharam2019unifying}, Anantharam, Jog and Nair  established the following entropic inequality:
\begin{theorem}\label{thm:AJN}
Consider independent random vectors $(Z_i)_{1\leq i \leq k}$ taking values in $(E_i)_{1\leq i\leq k}$,  respectively, each having density with respect to Lebesgue measure,  finite entropies, and finite second moments.    It holds that  
\begin{align}
\sum_{i=1}^kc_i  h\left( Z_i \right)  - \sum_{j=1}^m d_j  h\left(\sum_{i=1}^k B_{ij}  Z_i \right)  \leq M_g(\mathbf{c},\mathbf{d},\mathbf{B}), \label{AJNineq}
\end{align}
where the constant $M_g(\mathbf{c},\mathbf{d},\mathbf{B})$ is defined as the supremum of the LHS, taken over independent gaussian  $(Z_i)_{1\leq i \leq k}$.  Moreover,   the quantity $M_g(\mathbf{c},\mathbf{d},\mathbf{B})$ is finite if and only if we have the scaling condition
\begin{align}
\sum_{i=1}^k c_i \dim(E_i) = \sum_{j=1}^m d_j \dim(E^j) \label{CONDITIONscaling2}
\end{align}
and the dimension condition
\begin{align}
\sum_{i=1}^k c_i \dim(\pi_{E_i}T) \leq  \sum_{j=1}^m d_j \dim(B_j T) ~~\mbox{for all product-form subspaces $T\subseteq E_0$.} \label{CONDITIONdimension2}
\end{align}
\end{theorem}
The inequality \eqref{AJNineq} is of interest because it simultaneously expresses both the entropy power inequality and the (entropic formulation \cite{carlen2009subadditivity} of) the Brascamp-Lieb inequality.  The former is generally considered a consequence of the latter, obtained by considering a limiting case of parameters.  As such, an inequality encompassing both simultaneously was previously not known.   Analogously, it turns out that Theorem \ref{thm:AJN} can  be derived as a corollary of Theorem \ref{thm:FRBL} by considering a limiting case of parameters. 

To give the argument, we first state an entropic characterization of $D(\mathbf{c},\mathbf{d},\mathbf{B})$, which directly parallels Proposition \ref{prop:entropyCharacterizationDg} (here, the reader is reminded of the notation introduced in Section \ref{subsec:EntropyDg}).   Specifically, the following entropic characterization of $D(\mathbf{c},\mathbf{d},\mathbf{B})$ is a special case of \cite[Theorem 1]{liu2018forward}, which generalizes to abstract settings and extends the entropic formulation of the forward Brascamp-Lieb inequality  due to Carlen and Cordero-Erasquin  \cite{carlen2009subadditivity}, as well as the entropic formulation of the reverse Brascamp-Lieb inequality independently put forth in \cite{liu2016brascamp} and \cite{beigi2016equivalent} (the latter being specific to discrete spaces).   
\begin{theorem}\label{thm:entDualGeneral}
If $(Z_i)_{1\leq i \leq k}$ are compactly supported random vectors in $(E_i)_{1\leq i\leq k}$,  respectively, each having density with respect to Lebesgue measure and finite entropies,  then 
\begin{align}
\sum_{i=1}^kc_i  h\left( c_i^{-1} Z_i \right)  - \max_{P_{\mathbf{Z}}\in \Pi(P_{Z_1}, \dots, P_{Z_k})}  \sum_{j=1}^m d_j h\left(\sum_{i=1}^k B_{ij} Z_i \right) \leq D(\mathbf{c},\mathbf{d},\mathbf{B}),\label{compactEntropyD}
\end{align}
where the (always attained) maximum is over all couplings of the $(Z_i)_{1\leq i \leq k}$.  Moreover, the constant $D(\mathbf{c},\mathbf{d},\mathbf{B})$ is best possible.  
\end{theorem}
Theorem \ref{thm:entDualGeneral} is proved similarly to Theorem \ref{thm:equivFormsDg}, except that the Fenchel-Rockafellar theorem is applied to the topological vector space $X=C_c(E_0)$, and equivalence is shown to the functional formulation of $D(\mathbf{c},\mathbf{d},\mathbf{B})$.  Readers can fill in the details as an exercise, or refer to the proof of the more general \cite[Theorem 1]{liu2018forward}. Despite the  similarity of statements and proof strategies, it is not immediate to derive Proposition \ref{prop:entropyCharacterizationDg} as a special case due to several subtle technical issues that need to be dealt with.

The following proof  provides a nice example of where the entropic characterization of $D(\mathbf{c},\mathbf{d},\mathbf{B})$  can be useful.

\begin{proof}[Proof of Theorem \ref{thm:AJN}] 
First, we note that  \eqref{CONDITIONscaling2} and \eqref{CONDITIONdimension2} are necessary conditions for finiteness, which can be checked by testing on gaussian $(Z_i)_{1\leq i\leq k}$ which put different variances in directions $\pi_{E_i}T$ and $(\pi_{E_i}T)^{\perp}$. So, we assume henceforth that \eqref{CONDITIONscaling2} and \eqref{CONDITIONdimension2} hold.

Let $M(\mathbf{c},\mathbf{d},\mathbf{B})$ denote the supremum of the LHS of \eqref{AJNineq} over all independent $(Z_i)_{1\leq i \leq k}$ with finite entropies and finite second moments.  Note that in taking this supremum, it suffices to consider compactly supported $(Z_i)_{1\leq i \leq k}$.  Indeed, if $(Z_i)_{1\leq i \leq k}$ are not compactly supported and have finite second moments, then letting $Z_{i,R}$ be the restriction of $Z_i$ to the ball of radius $R$, we have $\lim_{R\to\infty} h(Z_{i,R}) = h(Z_{i})$ by dominated convergence, and $\limsup_{R\to\infty} h\left(\sum_{i=1}^k B_{ij} Z_{i,R} \right) \leq  h\left(\sum_{i=1}^k B_{ij} Z_{i} \right)$ by weak upper semicontinuity of Shannon entropy under a second moment constraint.    %

We will consider an application of the forward-reverse Brascamp-Lieb inequality with modified coefficients $c'_i = (c_i+t)$, $1\leq i\leq k$, $d'_j = d_j$, $1\leq j\leq m$ and augmented datum having $E^{m+1}:=E_0$, $B_{m+1}:=\operatorname{id}_{E_0}$ and $d'_{m+1}:=t$, where $t$ is a parameter that will tend to $+\infty$.  Denote this augmented datum by $(\mathbf{c}+t, (\mathbf{d}, t), \mathbf{B}\cup\{B_{m+1}\})$.  %

Considering  independent gaussian $(Z_i)_{1\leq i \leq k}$, let $D(P\|Q) := \int dP \log \left(\frac{dP}{dQ}\right)\geq 0$ denote the relative entropy between probability measures $P$ and $Q$ satisfying $P\ll Q$, and for any  $t \geq 0$ observe
\begin{align*}
&\sum_{i=1}^k(c_i  +t) h\left( Z_i \right)  - \sup_{P_{\tilde{\mathbf{Z}}}\in \Pi(P_{Z_1},\dots, P_{Z_k}) }\left( \sum_{j=1}^m d_j h\left(\sum_{i=1}^k B_{ij} \tilde Z_i \right)  + t h\left( \tilde{Z}_1, \dots, \tilde Z_k \right) \right) \\
&=\sum_{i=1}^k c_i   h\left( Z_i \right)  - \sup_{P_{\tilde{\mathbf{Z}}}\in \Pi(P_{Z_1},\dots, P_{Z_k}) }\left( \sum_{j=1}^m d_j h\left(\sum_{i=1}^k B_{ij} \tilde Z_i \right)  - t D\left( P_{\tilde{\mathbf{Z}}}\| P_{Z_1}\times \cdots\times P_{Z_k}  \right)   \right) \\
&\leq \sum_{i=1}^kc_i  h\left( Z_i \right)  - \sum_{j=1}^m d_j h\left(\sum_{i=1}^k B_{ij}  Z_i \right) ,
\end{align*}
where the equality follows by definition of relative entropy, and the inequality follows by considering the independent coupling as one element in the set the supremum is taken over.   By definitions and Proposition \ref{prop:entropyCharacterizationDg}, we conclude 
\begin{align}
D_g(\mathbf{c}+t, (\mathbf{d}, t), \mathbf{B}\cup\{B_{m+1}\}) + \sum_{i=1}^k (c_i+t)\dim(E_i) \log (c_i+t) \leq M_g(\mathbf{c}, \mathbf{d}, \mathbf{B})~~\mbox{for all $t\geq 0$.}\label{DgMgupperBound} 
\end{align}
As for finiteness of $M_g(\mathbf{c},\mathbf{d},\mathbf{B})$, this was already taken care of in the proof of Proposition \ref{prop:sufficientCondtions}.  Indeed, using the  the relaxation \eqref{eq:boundMaxEntropyByIndep}, we established finiteness of $D_g(\mathbf{c},\mathbf{d},\mathbf{B})$ by showing
\begin{align*}
\sum_{i=1}^k c_i   h\left( Z_i \right)  -  \sup_{P_{\tilde{\mathbf{Z}}}\in \Pi(P_{Z_1},\dots, P_{Z_k}) } \sum_{j=1}^m d_j h\left(\sum_{i=1}^k B_{ij}   \tilde Z_i \right)
\leq \sum_{i=1}^k c_i   h\left( Z_i \right)  -  \sum_{j=1}^m d_j h\left(\sum_{i=1}^k B_{ij}   Z_i \right) \leq C <+\infty
\end{align*}
for independent gaussian $(Z_i)_{1\leq i \leq k}$ when  \eqref{CONDITIONscaling2} and \eqref{CONDITIONdimension2} hold.

Now, fix arbitrary $(Z_i)_{1\leq i \leq k}$ with compact support and  finite entropies, and  for each $n\geq 0$, consider a coupling $(Z^{(n)}_1, \dots, Z^{(n)}_k)\sim P_{\mathbf{Z}^{(n)}} \in  \Pi(P_{Z_1},\dots, P_{Z_k}) $ satisfying 
\begin{align*}
&  \sum_{j=1}^m d_j h\left(\sum_{i=1}^k B_{ij}  Z^{(n)}_i \right)  - n D\left( P_{\mathbf{Z}^{(n)}}  \| P_{Z_1}\times \cdots\times P_{Z_k}  \right)  
\\
&\geq 
  \sup_{P_{\tilde{\mathbf{Z}}}\in \Pi(P_{Z_1},\dots, P_{Z_k}) }\left( \sum_{j=1}^m d_j h\left( \sum_{i=1}^k B_{ij} \tilde Z_i \right)  - n D\left( P_{\tilde{\mathbf{Z}}}\| P_{Z_1}\times \cdots\times P_{Z_k}  \right)    \right) - \frac{1}{n}.
\end{align*}
Since the RHS is bounded from below by selecting the independent coupling and the entropies $h\left( \sum_{i=1}^k B_{ij}  Z^{(n)}_i \right)$ can be uniformly bounded from above in terms of the second moments of the marginals $(Z_i)_{1\leq i \leq k}$, it is clear that we must have 
$$
\lim_{n\to\infty} n D\left( P_{\mathbf{Z}^{(n)}}  \| P_{Z_1}\times \cdots\times P_{Z_k}  \right)=0.
$$
In particular, by  weak upper semicontinuity of Shannon entropy under second moment constraint, we conclude
\begin{align*} 
\limsup_{n\to\infty}   \sup_{P_{\tilde{\mathbf{Z}}}\in \Pi(P_{Z_1},\dots, P_{Z_k}) }\left( \sum_{j=1}^m d_j h\left( \sum_{i=1}^k B_{ij} \tilde Z_i \right)  - n D\left( P_{\tilde{\mathbf{Z}}}\| P_{Z_1}\times \cdots\times P_{Z_k}  \right)    \right) \leq \sum_{j=1}^m d_j h\left(\sum_{i=1}^k B_{ij}  Z_i \right).
\end{align*}
Combining definitions with the above, Theorem \ref{thm:entDualGeneral}, and Theorem \ref{thm:FRBL}, we have
\begin{align*}
&\liminf_{n\to\infty} \left( D_g(\mathbf{c}+n, (\mathbf{d}, n), \mathbf{B}\cup\{B_{m+1}\}) + \sum_{i=1}^k (c_i+n)\dim(E_i) \log (c_i+n)  \right) \\
&\geq \sum_{i=1}^k c_i   h\left( Z_i \right)  - \limsup_{n\to\infty}  \sup_{P_{\tilde{\mathbf{Z}}}\in \Pi(P_{Z_1},\dots, P_{Z_k}) }\left( \sum_{j=1}^m d_j h\left(\sum_{i=1}^k B_{ij} \tilde Z_i \right)  - n D\left( P_{\tilde{\mathbf{Z}}}\| P_{Z_1}\times \cdots\times P_{Z_k}  \right)   \right) \notag\\
&\geq \sum_{i=1}^k c_i   h\left( Z_i \right)  -  \sum_{j=1}^m d_j h\left(\sum_{i=1}^k B_{ij}   Z_i \right).
\end{align*}
Since the  $(Z_i)_{1\leq i \leq k}$ were arbitrary, it follows from \eqref{DgMgupperBound} and the subsequent remarks that    
$$M(\mathbf{c},\mathbf{d},\mathbf{B})\leq M_g(\mathbf{c},\mathbf{d},\mathbf{B}) <+\infty,$$
completing the proof.
\end{proof}

\appendix
\section{Definition and basic properties of Shannon entropy}\label{app:Entropy}
This appendix provides an overview of the basic definitions and properties of Shannon entropy, with the goal of allowing an unfamiliar reader to follow the entropy computations in the body of this manuscript.  In the interest of keeping things brief and self-contained, we focus on the case of random vectors in a Euclidean space with finite second moments, since this is sufficient for our purposes.  The interested reader can find a  general treatment in   \cite[Chapter 5]{GrayIT}.

Let $E$ be a Euclidean space, and assume $X$ is a random vector taking values in $E$, having  density $f_X : E\to [0,+\infty)$ with respect to Lebesgue measure.  The Shannon entropy (referred to henceforth as simply \emph{entropy}) associated to $X$ is defined as
$$
 h(X) := -\int_E f_X(x) \log f_X(x) dx.
$$
The entropy is said to exist if the integral is well-defined in the Lebesgue sense.  Here, we adopt the convention that $0\cdot \log 0 = 0$.   The reader should note that the entropy $h(X)$ is a functional of the density $f_X$, and is not a function of the realization of the random vector $X$.    If $X$ does not have density with respect to Lebesgue measure, we adopt the convention that $h(X)=-\infty$. 

By a simple change of variables, we have the following elementary property of entropy:
\begin{proposition}[Scaling property  for Shannon entropy]\label{app:scaling}
If  the entropy of $X$ exists and $A: E \to E$ is an invertible linear transformation, then we have the scaling property
$$
h(A X) = \log |\det A| + h(X).
$$
\end{proposition}

In this paper, we work exclusively with random vectors having finite second moments.  In this context, we note that the inequality $-\log(z) \geq 1-z$ on $z\in [0,+\infty)$ yields
$$
0 \leq \int_E f_X(x) \log \frac{f_X(x)}{\phi(x)} dx,
$$
where $\phi(x):=(2\pi)^{-\dim(E)/2} e^{-|x|^2/2}$ is the standard gaussian density on $E$.  We may conclude:
\begin{proposition}[Entropy of random vectors with bounded second moments]\label{app:ExistEntropy}
Let $X$ have density on $E$.  If $\EE|X|^2 <\infty$, then the entropy $h(X)$ exists, and is bounded from above as
$$h(X) \leq \frac{1}{2}\log \left((2\pi e)^{\dim(E)}\EE|X|^2\right).$$
\end{proposition}

If $E_1,E_2$ are Euclidean spaces and $(X,Y)$ is a pair of random vectors taking values in $E_1\times E_2$ with  joint density $f_{XY}: E_1\times E_2 \to  [0,+\infty)$, the \emph{joint entropy} of $(X,Y)$ is defined as 
$$
h(X,Y) = -\iint_{E_1\times E_2}  f_{XY}(x,y)\log f_{XY}(x,y) dx dy.
$$
Of course, this is consistent with the original definition of entropy, applied to  the $(E_1\times E_2)$-valued  random vector $(X,Y)$.    A simple property is as follows:
\begin{proposition}[Subaddivity of entropy]\label{app:Subadditive}
Let $(X,Y)$ be a pair of random vectors taking values in $E_1\times E_2$ with  joint density $f_{XY}: E_1\times E_2 \to  [0,+\infty)$.  If $(X,Y)$ have finite second moments, then the joint and marginal entropies exist and satisfy 
$$
h(X,Y) \leq h(X) + h(Y). 
$$
This is met with equality if $X,Y$ are independent. 
\end{proposition}
\begin{proof}
Since all entropies exist and are bounded from above, we may assume $h(X,Y)$ is finite, else the claim is trivial.  Similar to before, we note the following inequality:
$$
0 \leq \iint_{E_1\times E_2}  f_{XY}(x,y)\log \frac{f_{XY}(x,y)}{f_X(x) f_Y(y)} dx dy.
$$
Using existence and finiteness of the joint entropy, we apply linearity of the integral to conclude 
\begin{align}
-\iint_{E_1\times E_2}  f_{XY}(x,y)\log \left(  {f_{X}(x)}   {f_{Y}(y)} \right) dx dy\geq h(X,Y). \label{prelimSubadditive}
\end{align}
Now, let us decompose the density
$$
f_{XY}(x,y) = f_X(x) f_{Y|X}(y|x),~~~~x\in E_1, y\in E_2,  
$$
where, for $x\in E_1$, the function $f_{Y|X}(\cdot | x)$ denotes the density of a conditional regular probability $P_{Y|X=x}$ on $E_2$, and coincides with the ratio $f_{XY}(x,y)/f_X(x)$ for almost every $x$ in the support of $f_X$.   Similar to how we argued entropy was bounded from above under a second moment constraint, we may show that the integral
$$
-\iint_{E_1\times E_2}  f_{XY}(x,y)\log \left(  {f_{X}(x)}  \right) dx dy  
$$
exists and is bounded from above.  The same conclusion holds with $f_X(x)$ replaced by $f_Y(y)$.   From this and \eqref{prelimSubadditive}, we conclude using the assumed finiteness of $h(X,Y)$ that the function $(x,y) \longmapsto f_{XY}(x,y)\log \left(  {f_{X}(x)}  \right)$ is integrable, and is equal to $-h(X)$ by the Fubini-Tonelli theorem.  Similar for $h(Y)$, so linearity of the integral applied to \eqref{prelimSubadditive} proves the claim.  
\end{proof}

Finally, we note the following consequence of lower semicontinuity of relative entropy, which follows from the Donsker-Varadhan variational formula.  See, e.g., \cite[Lemma A2]{liu2018forward} for details.
\begin{proposition}[Upper semicontinuity of entropy under second moment constraint]\label{app:WeakUSC}
Let $(X_n)_{n\geq 1}$ be a sequence of random vectors on a Euclidean space $E$, such that $X_n \to X$ weakly.  If $\sup_{n\geq 1} \EE|X_n|^2<+\infty$, then 
$$
\limsup_{n\to\infty} h(X_n) \leq h(X). 
$$
\end{proposition}

\end{document}